\documentclass{amsart}
\usepackage{amsmath}
\usepackage{amssymb}
\usepackage{latexsym}
 \numberwithin{equation}{section}
\newtheorem{theorem}{Theorem}%[]

\newtheorem{thm}[theorem]{Theorem}
\newtheorem{cor}[theorem]{Corollary}
\newtheorem{lem}[theorem]{Lemma}
\newtheorem{prop}[theorem]{Proposition}
\theoremstyle{definition}
\newtheorem{defn}[theorem]{Definition}

\theoremstyle{remark}
\newtheorem{remark}[theorem]{Remark}
\renewenvironment{proof}[1][\proofname]{\par
  \normalfont
  \trivlist
  \item[\hskip\labelsep\itshape
    \bfseries{#1.}]\ignorespaces
}{%
  \qed\endtrivlist
}

\newcommand{\RR}{\mathbb R}
\newcommand{\NN}{\mathbb N}

\newcommand{\eps}{\varepsilon}
\newcommand{\lam}{\lambda}

\newcommand{\scp}[2]{\left\langle #1,#2\right\rangle}    
\newcommand{\bscp}[2]{\big\langle #1,#2\big\rangle}     
\newcommand{\abs}[1]{\left\vert#1\right\vert}

\newcommand{\nabs}[1]{\vert#1\vert}
\newcommand{\bigabs}[1]{\big\vert#1\big\vert}

\newcommand{\norm}[1]{\left\|#1\right\|}

\newcommand{\mysetr}[2] {\left\{#1\,\left|\,#2\right.\right\}}
\newcommand{\mysetl}[2] {\left\{\left.#1\,\right|\,#2\right\}}

\newcommand{\To}{\longrightarrow}
\newcommand{\mbf}[1]{\boldsymbol{#1}}
\newcommand{\cU}{\mathcal{U}}
\newcommand{\cV}{\mathcal{V}}
\newcommand{\loc}{\text{loc}}
\newcommand{\s}{{\sigma}}

\newcommand{\R}{{\mathbb R}}
\newcommand{\N}{{\mathbb N}}
\newcommand{\cR}{\mathcal R}
\newcommand{\ups}{\Upsilon_{{\mathcal R}}^p(\R^{M\times N})}
\newcommand{\DM}{{\rm DM}^p_{\mathcal R}(\O;\R^{M\times N})}
\newcommand{\DMS}{{\rm DM}^p_{\mathcal R,S}(\O;\R^{M\times N})}
\newcommand{\BGDM}{{\rm BGDM}^p_{\mathcal R}(\O;\R^{M\times N})}
\newcommand{\GDM}{{\rm GDM}^p_{\mathcal R}(\O;\R^{M\times N})}
\newcommand{\be}{\begin{eqnarray}}
\newcommand{\com}{{\beta_{{\mathcal R}}\R^{M\times N}}}
\newcommand{\ee}{\end{eqnarray}}
\renewcommand{\O}{\Omega}
\newcommand{\rca}{{\rm rca}}
\newcommand{\md}{{\rm d}}
\newcommand{\wto}{{\rightharpoonup}}
\begin{document}
\begin{sloppypar}

\title[Oscillations and concentrations up to the boundary]%
{Oscillations and concentrations in sequences of gradients up to the boundary}\thanks{ The work of M.K. was
supported by the grants  IAA 1075402 (GA AV \v{C}R), P201/10/0357 (GA \v{C}R), and
VZ6840770021 (M\v{S}MT \v{C}R). }

%draft, last edit by SK, 2011/07/15

\author{Stefan Kr\"omer}
\address{Mathematisches Institut, Universit\"at zu K\"oln, 50923 K\"oln, Germany\\
{\tt skroemer@math.uni-koeln.de}}
\author{Martin Kru\v{z}\'{\i}k}
\address{Institute of Information Theory and Automation,
Academy of Sciences of the Czech Republic, Pod vod\'{a}renskou
v\v{e}\v{z}\'{\i}~4, CZ-182~08~Praha~8, Czech Republic (corresponding
address) \& Faculty of Civil Engineering, Czech Technical
University, Th\'{a}kurova 7, CZ-166~ 29~Praha~6, Czech Republic\\  {\tt
kruzik@utia.cas.cz} }

\begin{abstract}
Oscillations and concentrations in sequences of
gradients $\{\nabla u_k\}$, bounded in $L^p(\O;\R^{M\times N})$
 if $p>1$ and $\O\subset\R^n$ is a bounded domain with
the  extension property in $W^{1,p}$, and their interaction with local integral functionals
can be described by a generalization of Young measures due to DiPerna and Majda.
We characterize such DiPerna-Majda measures, thereby extending a result by Ka\l amajska and Kru\v{z}\'{\i}k \cite{KaKru08a}, 
where the full characterization was possible only for sequences subject to a fixed Dirichlet boundary condition. As an application we state a relaxation result 
for noncoercive multiple-integral functionals. 
\end{abstract}

%\begin{resume} ... \end{resume}
%

\maketitle

\section{Introduction}
Oscillations and/or concentrations in weakly convergent sequences appear in many problems in the calculus of variations,  partial differential equations, or  optimal
control theory, which admit  only $L^p$ but not $L^\infty$ a priori estimates.
 Young measures \cite{y} successfully capture oscillatory behavior  of
sequences, however,  they completely miss concentrations. There are several available tools to deal with concentrations.
They can be considered as generalization of Young measures, see for example Alibert's and Bouchitt\'{e}'s approach  \cite{ab}, DiPerna's and Majda's treatment of concentrations \cite{DiPeMaj87a}, or  Fonseca's method described in \cite{fonseca}. An overview can be found in \cite{r,tartar1}.  Moreover, in many cases,
we are interested in oscillation/concentration effects generated by sequences of gradients. Oscillatory behavior of gradients was described by Kinderlehrer and Pedregal \cite{k-p1,k-p} in terms of  gradient  Young measures, cf.~also \cite{pedregal}.
 The  first attempt to characterize both
oscillations and concentrations in sequences of gradients is due
to Fonseca, M\"{u}ller, and Pedregal \cite{fmp}. They dealt with a special situation of $\{g(\cdot) v(\nabla u_k(\cdot))\}_{k\in\N}$ where $v$ is positively $p$-homogeneous, $u_k\in W^{1,p}(\O;\R^m)$, $p>1$, with  $g$ continuous and   vanishing on $\partial\O$. 
Later on, a characterization of oscillation/concentration effects in terms of  DiPerna's and Majda's generalization
of Young measures was given in \cite{KaKru08a} for arbitrary integrands and in \cite{ifmk} for sequences living in the kernel of a first-order differential operator. Recently, Kristensen and Rindler \cite{kristensen-rindler} characterized oscillation/concentration effects in the case $p=1$. 
Nevertheless, a complete  analysis of  boundary effects generated by gradients is still missing. We refer to 
\cite{KaKru08a} for the case where $u_k=u+W^{1,p}_0(\O;\R^M)$ on the boundary of the domain. As already observed by Meyers \cite{meyers}, concentration effects at the boundary are closely related to the sequential weak lower semicontinuity of integral functionals $I:W^{1,p}(\O;\R^m)\to\R$: $I(u)=\int_\O v(\nabla u(x))\,\md x$ where  $v:\R^{m\times n}\to\R$ is continuous and such that $|v|\le C(1+|\cdot|^p)$ for some constant $C>0$.  Recently, the first author \cite{Kroe10d} stated an integral necessary and sufficient condition ensuring weak lower semicontinuity in $W^{1,p}$ which is equivalent to the one of Meyers, however, much easier to handle due to its local character. We also refer to \cite{BaZh90} where the weak lower semicontinuity is treated using the so-called Biting Lemma \cite{BrCha80a}.  

The aim of this contribution is to give necessary and sufficient conditions ensuring that a given DiPerna-Majda measure is generated by gradients without any restrictions on the generating sequence. In particular, we state a relaxation result for noncoercive integral functionals, see Theorem~\ref{relaxation} extending results by Dacorogna \cite{Da89B}. 
Let us mention that for coercive variational problems, i.e., $I(u)=\int_\O v(\nabla u(x))\,\md x$ with $c(-1+|U|^p)\le v(U)\le C(1+|U|^p)$, $p>1$,  minimizing sequences do not exhibit concentrations. In particular, if $\{u_k\}_{k\in\N}\subset W^{1,p}(\O;\R^M)$ is bounded and minimizing for $I$ then $\{|u_k|^p\}_{k\in\N}$ is equiintegrable. This is a consequence of the so-called decomposition lemma proved in \cite{fmp} and in an earlier version in \cite{Kri94a}. However,  for different growth and coercivity conditions, for instance if $N=M$, $v$ is finite on invertible matrices and satisfies  $c(-1+|U|^p+|U^{-1}|^p)\le v(U)\le C(-1+|U|^p+|U^{-1}|^p)$, the corresponding decomposition lemma is not available and appearance of concentrations in minimizing sequences cannot be a priori excluded \cite{BBMKGP11a}. 
We emphasize  that the aforementioned growth and coercivity conditions are relevant in nonlinear elasticity where $U$ is the deformation gradient and  $U^{-1}$ belongs to the so-called Seth-Hill family of strain measures see e.g.~\cite{ CuRa91a,MiLa01a}. In particular, $v(U)\to\infty$ if $\det U\to 0$. Hence, DiPerna-Majda measures can serve as a suitable tool for relaxation. We also refer to \cite{KrRo99a} for optimal control problems exhibiting concentrations and for their relaxation in terms of these measures including numerical approximation and to \cite{LMP07a} for a mathematical model of debonding  where concentration effects appear, as well.

\section{Notation and preliminaries}
Let us  start with a few definitions and with an explanation of our basic notation.
Having a bounded domain $\O\subset\R^N$ we denote by $C(\O)$ the space of continuous functions from $\O$ into $\R$.  Its subspace $C_0(\O)$ consists of functions in $C(\O)$ whose support is contained in $\O$. 
We write ``$\gamma$-almost all'' or ``$\gamma$-a.e.'' if  we mean ``up to a set with $\gamma$-measure zero''. If $\gamma$ is the $N$-dimensional Lebesgue measure and $M\subset\R^N$ we omit writing $\gamma$ in the notation.
Furthermore, $W^{1,p}(\O;\R^M)$, $1\le p<+\infty$  denotes the usual space of measurable mappings which are together with
their first (distributional) derivatives integrable to the $p$-th power.
The support of a measure $\sigma\in\ {\rm rca}(\O)$ is a smallest closed set $S$
such that $\sigma(A)=0$ if $S\cap A=\emptyset$.
%Finally, if $\sigma\in\rca(S)$ we write
%$\sigma_s$ and $d_\sigma$ for the singular part and density  of $\sigma$ defined by the Lebesgue decomposition, respectively. 
We denote by `w-$\lim$' the weak limit and by $B_r(x_0)$ an open ball in $\R^N$ centered at $x_0$ and the radius $r>0$. Given a set $E$, we write $\chi_E$ for its characteristic function, i.e., $\chi_E=1$ on $E$ and $\chi_E=0$ on the complement of $E$. Moreover, if $E\subset \RR^N$ and $r>0$, we define the $r$-neighborhood of $E$ by $(E)_r:=\bigcup_{x\in E} B_r(x)$ . The dot product on $\R^N$ is defined as $a\cdot b:=\sum_{i=1}^N a_i b_i$, and analogously on $\R^{M\times N}$. Finally, if $a\in\R^M$ and $b\in\R^N$ then  $a\otimes b\in\R^{M\times N}$ with $(a\otimes b)_{ij}=a_i b_j$, and $\mathbb{I}$ denotes the identity matrix.

\subsection{Global assumptions}

Unless stated otherwise, the following is assumed throughout the article:
\begin{align*}
&\begin{aligned}
	&1<p<\infty,~~M\in\NN,~~N\in \NN~~\text{with}~~N\geq 2,
\end{aligned}\label{H1}\tag{H1}\\
&\begin{aligned}
	&\text{$\O\subset \RR^N$ is open and bounded with boundary of class $C^1$,}
\end{aligned}\label{H2}\tag{H2}
\intertext{and}
&\begin{aligned}
	&\text{$\cR$ is a ring of bounded, continuous functions $v_0:\RR^{M\times N}\to \RR$, such that}\\
	&\begin{alignedat}[t]{2}
	&\text{(i)} && \text{$\cR$ is a complete and separable subset of $L^\infty(\RR^{M\times N})$}, \\
	&\text{(ii)} && \text{$C_0(\RR^{M\times N})\subset \cR$ and $1\in \cR$,}\\
	&\text{(iii)} &~~& v_0(\,\cdot \,Q)\in \cR~~\text{for every $v_0\in\cR$ and every $Q\in SO(N)$, and}\\
	&\text{(iv)} && \text{\eqref{v0ltuc} holds for each $v_0\in\cR$,}
	\end{alignedat}
\end{aligned}\label{H3}\tag{H3}
\end{align*}
i.e., 
\begin{equation}\label{v0ltuc}
	\begin{aligned}
	&\text{there exists $\alpha=\alpha(v_0):[0,\infty)\to [0,\infty)$ continuous with $\alpha(0)=0$ s.t.}\\
	&\abs{v_0(s)-v_0(t)}\leq \alpha\Big(\frac{\abs{s-t}}{1+\abs{s}+\abs{t}}\Big)~~\text{for every $s,t\in \RR^{M\times N}$.}
	\end{aligned}
\end{equation}
\begin{remark}
Neither (ii) nor (iii) are real restrictions, since we can always 
extend a given ring to achieve this artificially. 
\end{remark}
\begin{remark} 
A nontrivial example for a function $v_0$ satisfying \eqref{v0ltuc} is 
$v_0(s):=\sin( \log (1+|s|^2))$, $s\in\RR^{M\times N}$. We will use \eqref{v0ltuc} usually in form 
of the equivalent condition \eqref{vLpL1uc} derived in Lemma~\ref{lem:uc} below. 
Without the technical assumption \eqref{vLpL1uc}, a lot of our arguments break down; in particular, it is then no longer clear if 
$p$-qscb integrands (see Definition~\ref{def:qscb} below) are still precisely those that give rise to functionals that are weakly lower semicontinuous along purely concentrating sequences,
which is the cornerstone of our discussion of the boundary.
\end{remark}

\subsection{DiPerna-Majda measures}

In the context of DiPerna-Majda measures, we rely on the notation listed below.
For more background information, the reader is referred to \cite{r,KaKru08a} and \cite{Kru10a}.
\begin{itemize}

	\item ${\ups}:=\mysetl{v:\RR^{M\times N}\to \RR}{v(s)=v_0(s)(1+\abs{s}^p)~~\text{for a $v_0\in \cR$}}$.
	
	\item ${\com}$ denotes the compactification of $\RR^{M\times N}$ corresponding to $\cR$, i.e., 
	a compact set into which $\RR^{M\times N}$ is embedded homeomorphically and densely, such that
	each $v_0\in \cR$ has a unique continuous extension onto $\com$.
	Since we assume $\cR$ to be separable, the topology of $\com$ is metrizable.
	For more details, the reader is referred to \cite{engelking}.
	
	\item 
	{\rm rca}$(S)$ denotes the set of regular countably additive set functions on the Borel $\s$-algebra on a metrizable set  
	$S$ (cf.~\cite{d-s}), and its subset {\rm rca}$^+_1(S)$ denotes regular probability measures on a set $S$.	
	%$\rca(S)$ is the set of signed Radon measures on a metric space $S$, i.e., 
	%the real-valued, inner regular, countably additive set functions on the Borel $\sigma$-algebra of $S$.
	
	\item For $v_0\in \cR$ and $\hat{\nu}\in \rca(\com)$, we write
	$$
		\begin{alignedat}{2}
			&\bscp{\hat{\nu}}{v_0}&&:=\int_\com v_0(s)\,\hat{\nu}(ds),\\
			&\bscp{\hat{\nu}}{v_0}_\infty&&:=\int_{\com\setminus \RR^{M\times N}} v_0(s)\,\hat{\nu}(ds).
		\end{alignedat}
	$$		
	
	\item ${\sigma_s}$ and ${d_\sigma}$, respectively, denote the singular part and the
	density of the absolutely continuous part of $\sigma\in \rca(\overline{\O})$, with respect to Lebesgue decomposition. 
	
	\item For $\sigma\in \rca(\overline{\O})$, the space ${L_w^\infty(\overline{\O},\sigma;\rca(\com))}$ consists of those 
	functions $x\mapsto \hat{\nu}_x$ which are weak$^*$-measurable (i.e., $x\mapsto \bscp{\hat{\nu}_x}{v_0}$ is 
	Borel measurable for every $v_0\in \cR$) and $\sigma$-essentially bounded.
	
	\item Let $1\leq p<\infty$, let $(U_n)\subset L^p(\O;\RR^{M\times N})$ be a bounded sequence,
	and let $\sigma\in \rca(\overline{\O})$ and $\hat{\nu}\in L_w^\infty(\overline{\O},\sigma;\rca(\com))$.
	We call $(\sigma,\hat{\nu})$ the \emph{DiPerna-Majda measure generated by $(U_n)$}, if
	\begin{equation}\label{DMgenerate}
		\int_\O \varphi(x)v(U_n(x))\,dx\underset{n\to\infty}{\To}
		\int_{\overline{\O}} \varphi(x)\, \bscp{\hat{\nu}_x}{v_0}\,\sigma(dx),
	\end{equation}
	for every $\varphi\in C(\overline{\O})$ and every $v_0\in\cR$, with $v(\cdot):=v_0(\cdot)(1+\abs{\cdot}^p)$.
	Every bounded sequence in $L^p(\O;\RR^{M\times N})$ has a subsequence which generates a DiPerna-Majda measure, 
	see \cite{DiPeMaj87a}.
	
	\item The set of all DiPerna-Majda-measures generated by a bounded sequence in $L^p(\O;\RR^{M\times N})$ 
	is denoted by ${\DM}$. 
	
	\item The set of all DiPerna-Majda-measures in $\DM$ generated by gradients, i.e., 
	by $(\nabla u_n)$ for a bounded sequence $(u_n)\subset W^{1,p}(\O;\RR^M)$, is denoted by ${\GDM}$.
	
\end{itemize}
In addition, we recall the following two general results on DiPerna-Majda-measures:
\begin{prop}[\cite{KruRou97a}]\label{prop:DMchar}
Let $1\leq p<\infty$, let $\O\subset\RR^N$ be a bounded open domain such that
$|\partial\O | =0$, let $\cR$ be a separable complete subring of
the ring of all continuous bounded functions on $\RR^{M\times N}$
and let $(\sigma,\hat{\nu})\in \rca(\overline{\O})\times L^{\infty}_{w}(\overline{\O},\sigma; \rca(\com))$.
Then $(\sigma,\hat{\nu})\in \DM$ if and only 
if all of the following conditions are satisfied:
\begin{enumerate}
\item[(i)] $\sigma\geq 0$;
\item[(ii)] $\bar{\sigma}\in \rca(\overline{\O})$, $\bar{\sigma}(dx):=\hat{\nu}_x(\RR^{M\times N})\,\sigma(dx)$,
is absolutely continuous with respect to the Lebesgue measure;
\item[(iii)] for a.a. $x\in\O$,
$$
	\qquad \hat{\nu}_x(\RR^{M\times N})>0,\quad\text{and}\quad 
	d_{\bar{\sigma}}(x)=\left(\int_{\RR^{M\times N}}
	\frac{\hat\nu_x(ds)}{1+\abs{s}^p}\right)^{-1} \hat{\nu}_x(\RR^{M\times N});
$$
\item[(iv)] for $\sigma$-a.a. $x\in\overline{\O}$, $\hat{\nu}_x\geq 0$ and
$\hat{\nu}_x(\com)=1$.
\end{enumerate}
\end{prop}

\begin{remark}\label{rem:DMprop}
Proposition~\ref{prop:DMchar} (ii) implies that for $\sigma_s$-a.e.~$x\in\overline{\O}$,
$\hat{\nu}_x(\RR^{M\times N})=0$.
In particular, $\hat{\nu}_x(\RR^{M\times N})=0$ for $\sigma$-a.e.~$x\in\partial\O$ (provided that $\abs{\partial\O}=0$), whence
$$
	\bscp{\hat{\nu}_x}{v_0}=\bscp{\hat{\nu}_x}{v_0}_\infty\quad\text{for $\sigma$-a.e.~$x\in\partial\O$ and every $v_0\in\cR$}.
$$
Moreover, as a consequence of (ii) and (iii), the density of the absolutely continuous part of $\sigma$ with respect to the 
Lebesgue measure is given by
\begin{equation}\label{dsigma}
	d_\sigma(x)=\left(\int_{\RR^{M\times N}}	\frac{\hat\nu_x(ds)}{1+\abs{s}^p}\right)^{-1}.
\end{equation}
\end{remark}
%\begin{lem}[e.g., see Proposition 4.1 in \cite{AlBou97a} or Proposition 3.2.17 in \cite{Rou97B}]
%\cite[Lemma~1., Th.~1,2]{k-r-dm},
%\cite[Prop.~3.2.17]{r}), \cite[Proposition 4.1, part
%(iii)]{ab} and \cite[Lemma 3.1, part (ii)]{ak2}]
%\label{lem:DMsingularpart}
%Let $1\leq p<\infty$, let $\O\subset\RR^N$ be a bounded open domain such that $\abs{\partial\O}=0$, 
%let $\cR$ be a separable complete subring of the ring of
%all continuous bounded functions on $\R^{m\times n}$ and let
%$(\sigma,\hat{\nu})\in \DM$. Then for $\sigma_s$-almost every $x\in\overline{\O}$,
%$\hat{\nu}_x(\RR^{M\times N})=0$.
%IS THE CASE $p=1$ OK?
%\end{lem}
%Moreover, $(\sigma,\hat{\nu})$ is uniquely determined by (the subsequence of) $U_n$. For proofs and more details, see \cite{DiPeMaj87a}. 

\subsection{Quasiconvexity and $p$-quasi-subcritical growth from below}

Two notions related to the weak lower semicontinuity of integrals functionals on $W^{1,p}$ play an important role in our main result.
The first one is the well-known quasiconvexity of {\sc Morrey} \cite{Mo52a}:
\begin{defn}[quasiconvexity and quasiconvex envelope, e.g.~see \cite{Da89B}]
We say that a function $f:\RR^{M\times N}\to \RR$
is \emph{quasiconvex} if for some bounded regular domain $\Lambda\subset \RR^N$,
the integrals below are defined and
$$
	\int_\Lambda f(s+\nabla\varphi(y))\,dy\geq \int_\Lambda f(s)\,dy
$$
for every $s\in \RR^{M\times N}$ and every $\varphi\in W_0^{1,\infty}(\Lambda;\RR^M)$. 
The \emph{quasiconvex envelope $Qf$ of $f$} is defined as the largest quasiconvex function below $f$, i.e., for $s\in \RR^{M\times N}$,
$$
	Qf(s):=\sup\mysetr{g(s)}{\text{$g:\RR^{M\times N}\to \RR$ is quasiconvex and $g \leq f$}},
$$
with $Qf\equiv -\infty$ if there is no admissible $g$.
\end{defn}
If $f$ is locally bounded, its quasiconvex envelope can be represented as
\begin{equation}\label{Qfrep}
	Qf(s)=\inf \mysetl{\frac{1}{\abs{\Lambda}}\int_\Lambda f(s+\nabla \varphi(y))\,dy}{\varphi\in W_0^{1,\infty}(\Lambda;\RR^M)},
\end{equation}
see \cite{Da89B}.

The following \emph{$p$-quasi-subcritical growth condition from below}, related to weak lower semicontinuity along purely concentrating sequences, 
first appeared in \cite{Kroe10d} (although the term $p$-qscb was not used for it there).
Its relevance comes from the fact that integral functionals of the form $u\mapsto \int_\O f(x,\nabla u(x))\,dx$ (assuming a $p$-growth condition and some smoothness) are wlsc in $W^{1,p}$ if and only if the integrand is quasiconvex and $p$-qscb, by the main result of \cite{Kroe10d}.
\begin{defn}[$p$-iqscb, $\nu$-$p$-bqscb, $p$-qscb at $x$, $p$-qscb]\label{def:qscb}
Let $p\in [1,\infty)$, and $f:\RR^{M\times N}\to \RR$ be continuous.
We say that $f$ is \emph{$p$-inner quasi-subcritical from below} (\emph{$p$-iqscb})
if
\begin{equation*}
	\begin{aligned}
	&\text{for every $\eps>0$, there exists $C_\eps\geq 0$ such that}\\
	&\int_{B_1} f(\nabla \varphi)\,dx\geq -\eps \int_{B_1} \abs{\nabla \varphi}^p\,dx-C_\eps~~
	\text{for every $\varphi \in W_0^{1,p}(B_1;\RR^M)$}.
	\end{aligned}
\end{equation*}
Given a unit vector $\nu\in \RR^N$, we say that $f$ is \emph{$\nu$-$p$-boundary quasi-subcritical from below} (\emph{$\nu$-$p$-bqscb})
if 
\begin{equation*}
\begin{aligned}
	&\text{for every $\eps>0$, there exists $C_\eps\geq 0$ such that}\\
	&\int_{D_\nu} f(\nabla \varphi)\,dx\geq -\eps \int_{D_\nu} \abs{\nabla \varphi}^p\,dx-C_\eps~~
	\text{for every $\varphi \in W^{1,p}_0(B_1;\RR^M)$}.
	%&\text{for every $\varphi \in W^{1,p}(D_\nu;\RR^M)$ with $\varphi=0$ on $\Gamma_\nu$}. %in the sense of trace
\end{aligned}
\end{equation*}
Here, $B_1=B_1(0)$ is the open unit ball in $\RR^N$ and $D_\nu:=\{x\in B_1\mid x\cdot\nu<0\}$.
Moreover, given an open, bounded set $\O\subset \RR^N$ with boundary of class $C^1$,
we say that $f$ is called \emph{$p$-quasi-subcritical from below} \emph{at $x_0\in \overline{\O}$} (\emph{$p$-qscb at $x_0$}),
if, in case $x_0\in \O$, $f$ is $p$-iqscb, and,
in case $x_0\in \partial\O$, $f$ is $\nu(x_0)$-$p$-bqscb, 
where $\nu(x_0)$ denotes the outer normal to $\partial\O$ at $x_0$.
Finally, we say that $f$ is \emph{$p$-quasi-subcritical from below} (\emph{$p$-qscb}) if 
$f$ is $p$-iqscb and $\nu$-$p$-bqscb for every $\nu\in S^{N-1}$. 
%(or, equivalently, if $f$ is $p$-qscb at every $x\in\bar{\O}$).
\end{defn}
\begin{remark}\label{rem:p-qscb_and_Q}
Quasiconvex functions are automatically $p$-iqscb. However, there exist functions that are $p$-qscb, but whose quasiconvex envelope is not.
Take, for instance,
$$
	f:\RR^{2\times 2}\to \RR,~~f(s):=\max\big\{ \det(s),~-\abs{s}^{\frac{3}{2}} \big\}.
$$
In this case, $f$ is $2$-qscb, while $Qf=\det$ (which is not $2$-qscb, see \cite{Kroe10d}):\\
\noindent{\bf $\mbf{f}$ is $\mbf{2}$-qscb:}
The trivial estimate $f(s)\geq -\abs{s}^{\frac{3}{2}}$ implies that for every $\eps>0$, there exists $C_\eps>0$ such that
$$
	f(s)\geq -\eps \abs{s}^2-C_\eps,~~\text{for every $s\in \RR^{2\times 2}$.}
$$
In particular, $f$ is $2$-qscb.\\
\noindent{\bf $\mbf{Qf\geq \det}$:}
Since $f(s)\geq \det(s)$, and the determinant is quasiconvex, we have that $Qf(s)\geq \det(s)$
for every $s\in \RR^{2\times 2}$. \\
\noindent{\bf $\mbf{Qf\leq \det}$:} Let $s=(s_1|s_2)\in \RR^{2\times 2}$, with $s_1$ and $s_2$ denoting the first and second column of $s$, respectively. 
If $\det(s)\geq 0$, $f(s)=\det(s)$ and thus $Qf(s)\leq \det(s)$. 
In particular, $Qf(0|s_2)\leq \det(0|s_2)$ for arbitrary $s_2\in \RR^{2\times 1}$. 
If $\det(s)< 0$, $f(hs_1|s_2)=\det(hs_1|s_2)$ whenever $h\geq 0$ is large enough, since $\det(hs_1|s_2) \abs{(hs_1|s_2)}^{-\frac{3}{2}}=O(h^{-\frac{1}{2}})\to 0$ as $h\to\infty$. 
Thus, $Qf(hs_1|s_2)\leq \det(hs_1|s_2)$ both if $h=0$ and if $h$ is large. 
Moreover, quasiconvexity implies rank-$1$-convexity, 
whence $Qf$ is convex along the line $h\mapsto (hs_1|s_2)$. 
Since the determinant is affine along this line, we infer that
$Qf(hs_1|s_2)\leq \det(hs_1|s_2)$ for every $h\in [0,\infty)$, and for $h=1$, this yields that 
$Qf(s)\leq \det(s)$.
\end{remark}

\section{Results}

Our main result characterizes DiPerna-Majda measures generated by gradients:
\begin{thm}\label{thm:GDMchar}
Assume that \eqref{H1}--\eqref{H3} hold,
and let $(\sigma,\hat\nu)\in \DM$. 
Then $(\sigma,\hat\nu) \in \GDM$
%, i.e., $(\sigma,\hat\nu)$ is generated by a sequence of gradients, 
%there is a bounded sequence $(u_k)\subset W^{1,p}(\O;\RR^M)$ such that $(\nabla u_k)$ generates $(\sigma,\hat\nu)$ 
if and only if the following four conditions are satisfied simultaneously:
\begin{enumerate}

\item[(i)] There exists $u\in W^{1,p}(\O;\RR^M)$ such that for a.e.~$x\in\O$,
\begin{equation*}%\label{GDMc-1}
	\nabla u(x)=d_\sigma(x)\int_{\com} \frac{s}{1+\abs{s}^p}\hat{\nu}_x(ds);
\end{equation*}

\item[(ii)] With $u$ from (i), for a.e.~$x\in\O$ and every $v\in \ups$,
\begin{equation*}%\label{GDMc-2}
	Qv(\nabla u(x))\leq d_\sigma(x)\int_{\com} \frac{v(s)}{1+\abs{s}^p}\hat{\nu}_x(ds);
\end{equation*}

\item[(iii)] For $\sigma$-a.e.~$x\in \O$ and every $v\in \ups$ such that $Q v>-\infty$,
\begin{equation*} %\label{GDMc-3i}
	0\leq \int_{\com\setminus \RR^{M\times N}} \frac{v(s)}{1+\abs{s}^p}\hat{\nu}_x(ds);
\end{equation*} 

\item[(iv)] For $\sigma$-a.e.~$x\in \partial\O$ and every $v\in \ups$ which is $p$-qscb at $x$,
\begin{equation*} %\label{GDMc-3b}
	0\leq \int_{\com\setminus \RR^{M\times N}} \frac{v(s)}{1+\abs{s}^p}\hat{\nu}_x(ds).
\end{equation*} 

\end{enumerate}
Here, $d_\sigma$ denotes the density of the absolutely continuous part of $\sigma$ with respect to the Lebesgue measure,
which is explicitly given by \eqref{dsigma}.
\end{thm}
The proof is the content of Section~\ref{sec:nec} and Section~\ref{sec:suff}.

The above theorem can be used to prove the following relaxation result similar to \cite[Th.~9.1, 9.8]{Da89B}.

\begin{thm}\label{relaxation}
Assume that \eqref{H1}--\eqref{H3} hold, let $h_0\in C(\bar\O\times\com)$, let $h(x,s):=h_0(x,s)(1+|s|^p)$ and
assume that $h(x,\cdot)$ is $p$-qscb at $x$ for all $x\in\partial\O$. 
For $u\in W^{1,p}(\O;\R^M)$ we define
$$
	H(u):=\int_\O h(x,\nabla u)\,dx\quad\text{and}\quad QH(u):=\int_\O Qh(x,\nabla u)\,dx,
$$
with the quasiconvex envelope $Qh(x,\cdot)$ of $h(x,\cdot)$.
Then the following holds:
\begin{itemize}
\item[(i)] If $u_n\wto u$ in $W^{1,p}(\O;\R^M)$, then 
\begin{equation*}%\label{RelaxGliminf}
\qquad\quad\liminf_{n\to\infty}H(u_n)\,dx\ge QH(u) .
%\qquad\quad\liminf_{n\to\infty}\int_\O h(x,\nabla u_n(x))\,dx\ge \int_\O Qh(x,\nabla u(x))\,dx\ .
\end{equation*}
%\item[(ii)] OLD, L^p RECOVERY \\
%Suppose that for all $x\in\bar\O$, $Qh(x,\cdot)>-\infty$.
%Then for every  $n\in\N$  and every  $\tilde u\in W^{1,p}(\O;\R^M)$ there is $\tilde u_n\in \tilde u+ W_0^{1,p}(\O;\R^M)$ such that 
%as $n\to\infty$, $\tilde{u}_n\to \tilde{u}$ in $L^{p}(\O;\R^M)$ and
%\begin{equation*}%\label{RelaxGlimsup}
%	\int_\O h(x,\nabla \tilde u_n(x))\,dx\to \int_\O Qh(x,\nabla \tilde u(x))\, dx. 
%\end{equation*}
\item[(ii)] 
For every $\eps>0$ and for every $\tilde u\in W^{1,p}(\O;\R^M)$, 
%such that $\int_\O Qh(x,\nabla \tilde u(x))\, dx>-\infty$
there exists a sequence $(\tilde u_n)\subset \tilde u+ W_0^{1,p}(\O;\R^M)$ such that 
$\tilde u_n\rightharpoonup \tilde u$ in $W^{1,p}(\O;\R^M)$, 
\begin{equation*}%\label{RelaxGlimsup}
%	\qquad\quad\limsup_{n\to\infty} H(\tilde{u}_n) \leq QH(\tilde{u})+\eps. 
	\qquad\quad\lim_{n\to\infty} \int_\O h(x,\nabla \tilde u_n)\,dx\leq 
	\left\{\begin{alignedat}[c]{2}
		&\int_\O Qh(x,\nabla \tilde u)\, dx+\eps &\quad&\text{if $\abs{E}=0$,}\\
		&-\eps^{-1}&\quad&\text{if $\abs{E}>0$,}
	\end{alignedat}\right.
\end{equation*}
where $E:=\mysetr{x\in\O}{Qh(x,\cdot)\equiv -\infty}$.
\end{itemize}
\end{thm}
The proof is given in Section~\ref{sec:relax}.

\vspace*{1ex}

\begin{remark} ~\\ \vspace*{-2ex}
\begin{itemize}
\item[(i)] Theorem~\ref{relaxation} implies that $\inf H=\inf QH$ on $W^{1,p}(\O;\RR^M)$.
\item[(ii)] By Theorem~\ref{relaxation} (ii), $QH$ is an upper bound for the sequentially weakly lower semicontinuous (swlsc) envelope of $H$ in $W^{1,p}(\O;\RR^M)$.
Hence, if we \emph{assume} that $QH$ is swlsc, then 
$QH$ is the swlsc envelope of $H$. 
However, it may happen that $QH$ is not swlsc.
Of course, $Qh$ is always quasiconvex, but even if it is a fairly regular finite-valued function, it can fail to be $p$-qscb as illustrated in Remark~\ref{rem:p-qscb_and_Q}. 
\item[(iii)] In Theorem~\ref{relaxation} (ii), it is not always possible to obtain an ``exact'' recovery sequence, corresponding to $\eps=0$.
However, this phenomenon can only occur if we do not have $p$-coercivity, cf.~\cite[Ex.~9.3 and Th.~9.8]{Da89B}.
\item[(iv)] If $h(x,\cdot)$ 
is not $p$-qscb at some point $x_0\in\overline\O$,
then the swlsc envelope of $H$ in $W^{1,p}(\Omega;\RR^M)$ is identically $-\infty$.
More precisely, for every $u\in W^{1,p}(\O;\RR^M)$ 
and every $K\in\NN$, there exists a bounded sequence $(u_n)\subset  W^{1,p}(\O;\RR^M)$ such that
the support of $u_n-u$ in $\overline{\O}$ shrinks to $x_0$ (in particular, $u_n\rightharpoonup u$) and $\lim_{n\to\infty} H(u_n)\leq -K$.
This can be seen following the proof of \cite[Proposition 3.8]{Kroe10d}\footnotemark .
\end{itemize}%
\end{remark}%
\footnotetext{One has to change the dilation constant $\alpha_n$ employed there by a fixed factor, to obtain $\norm{\nabla u_n}_{L^p}=\frac{K}{\eps}+\frac{1}{2}$ instead of $\norm{\nabla u_n}_{L^p}=1$ (with our $K$ and $\eps$ from the context in \cite{Kroe10d}).}

\section{Auxiliary results for concentrating sequences and $p$-qscb functions}

A key problem for us is the treatment of non-affine parts of the boundary. Of course, we can use local maps to transform a neighborhood of a boundary point into a situation with locally affine boundary. However, in expressions involving nonlinear integrands $v$ (or $f$, as in the definition of $p$-qscb) and non-compact sets of test functions or sequences with concentrations, this introduces an error that (as far as we understand) cannot be controlled without suitable uniform continuity properties of $v$. In \cite{Kroe10d}, a $p$-Lipschitz condition was used for this purpose, but here, we rely on the more general property \eqref{vLpL1uc} related to our assumption \eqref{v0ltuc} in \eqref{H3} as follows:
\begin{lem} \label{lem:uc}
Let $1\leq p<\infty$, let $v_0:\RR^{M\times N}\to \RR$ be continuous and bounded, and let $v(s):=v_0(s)(1+\abs{s}^p)$ for $s\in \RR^{M\times N}$. Then \eqref{v0ltuc} holds if and only if
\begin{equation}
\begin{aligned}\label{vQuc}
	&\text{there exists $\beta:[0,\infty)\to [0,\infty)$ continuous with $\beta(0)=0$ such that}\\
	&\abs{v(s)- v(t)}\leq \beta\Big(\frac{\abs{s-t}}{1+\abs{s}+\abs{t}}\Big)(1+\abs{s}^p+\abs{t}^p)\\
	&\text{for every $s\in\RR^{M\times N}$ and every $Q\in\RR^{N\times N}$.}
\end{aligned}
\end{equation}
Moreover, if $\Lambda\subset \RR^N$ is measurable with $0<\abs{\Lambda}<\infty$, then \eqref{vQuc} is equivalent to the following uniform continuity of 
the Nemytskii operator $U\mapsto v\circ U$, $L^p\to L^1$, on bounded subsets of $L^p(\Lambda;\RR^{M\times N})$:
\begin{equation}
\begin{aligned}\label{vLpL1uc}
	&\text{there exists $\gamma:[0,\infty)\to [0,\infty)$ continuous with $\gamma(0)=0$ such that}\\
	&\norm{v\circ U-v\circ W}_{L^1}\leq \gamma\big(\norm{U-W}_{L^p}\big)\big(1+\norm{U}_{L^p}^p+\norm{W}_{L^p}^p\big),\\
	&\text{with all norms taken over $\Lambda$, for every $U,W \in L^p(\Lambda;\RR^{M\times N})$.}
	%	
	%&\text{$\mbf{v}:L^p(\Lambda;\RR^{M\times N})\to L^1(\Lambda)$, $U\mapsto \mbf{v}(U):=v\circ U$,} \\
	%&\text{is uniformly continuous on bounded subsets of $L^p(\Lambda;\RR^{M\times N})$.}
\end{aligned}
\end{equation}
\end{lem}
\begin{remark}
For instance, both \eqref{v0ltuc} and \eqref{vLpL1uc} hold if either $v$ is $p$-Lipschitz or $\lim_{\abs{s}\to\infty} v_0(s)=0$. 
\end{remark}
\begin{proof}[Proof of Lemma~\ref{lem:uc}]~\\
{\bf \eqref{v0ltuc} implies \eqref{vQuc}:} Given \eqref{v0ltuc}, we have that
\begin{align*}
	&\abs{v(s)- v(t)} \\
	&\leq 
	\abs{v_0(s)- v_0(t)} (1+\abs{s}^p)
	+\abs{v_0(t)}\bigabs{\abs{t}^p-\abs{s}^p}\\
	&\leq \alpha\Big(\frac{\abs{s-t}}{1+\abs{s}+\abs{t}}\Big)(1+\abs{s}^p)
	+C\frac{\abs{s-t}}{1+\abs{s}+\abs{t}}	(1+\abs{s}^p+t^p)
\end{align*}
for some constant $C$, whence \eqref{vQuc} holds with $\beta(\delta):=\alpha(\delta)+C\delta$.\\
{\bf \eqref{vQuc} implies \eqref{v0ltuc}:} Given \eqref{vLpL1uc}, we have that 
\begin{align*}
	&\abs{v_0(s)- v_0(s Q)}(1+\abs{s}^p)\\
	&\leq \bigabs{v(s)- v(s Q)}+\abs{v_0(sQ)}\bigabs{\abs{sQ}^p-\abs{s}^p}\\
	&\leq \beta(\abs{I-Q}) (1+\abs{s}^p)+C \abs{I-Q} (1+\abs{s}^p)
\end{align*}
for some constant $C$, whence \eqref{v0ltuc} holds with $\alpha(t):=\beta(t)+Ct$.\\
{\bf \eqref{vLpL1uc} implies \eqref{vQuc}:} 
Since $v$ is uniformly continuous on bounded sets, it suffices to show \eqref{vQuc} for $1+\abs{s}+\abs{t}\geq \abs{\Lambda}^{-\frac{1}{p}}$.
Let $\Lambda_{s,t}\subset \Lambda$ be a subset of measure $\abs{\Lambda_{s,t}}=\frac{1}{(1+\abs{s}+\abs{t})^p}$.
By choosing $U(x):=s \chi_{\Lambda_{s,t}}(x)$ and $W(x):=s Q \chi_{\Lambda_{s,t}}(x)$, \eqref{vLpL1uc} yields that
\begin{align*}
	\abs{\Lambda_{s,t}} \abs{v(s)-v(t)}
	&\leq \gamma\big(\abs{\Lambda_{s,t}}^\frac{1}{p} \abs{s-t}\big) (1+\abs{\Lambda_{s,t}}\abs{s}^p+\abs{\Lambda_{s,t}}\abs{t}^p),
\end{align*}
and since $\abs{\Lambda_{s,t}}^{-1}=(1+\abs{s}+\abs{t})^p\leq 3^p (1+\abs{s}^p+\abs{t}^p)$, this implies \eqref{vQuc} with $\beta:=(3^p+1)\gamma$.\\
{\bf \eqref{vQuc} implies \eqref{vLpL1uc}:} 
Let $U,W\in L^p(\Lambda;\RR^{M\times N})$, let 
$$
	\Lambda_1:=\mysetr{x\in\Lambda}{\frac{\abs{U(x)-W(x)}}{1+\abs{U(x)}+\abs{W(x)}}>\norm{U-W}_{L^p(\Lambda;\RR^{M\times N})}^\frac{1}{2}}
$$
and let $\Lambda_2:=\Lambda\setminus \Lambda_1$. W.l.o.g., we may assume that $\beta$ is nondecreasing. By applying \eqref{vQuc} 
under the integral, we thus get that 
\begin{align*}
	&\int_{\Lambda_1} \abs{v(U(x))-v(W(x))}\,dx\\
	&\leq \int_{\Lambda_1} \beta\Big(\frac{\abs{U(x)-W(x)}}{1+\abs{U(x)}+\abs{W(x)}}\Big)(1+\abs{U(x)}^p+\abs{W(x)}^p)\,dx\\
	&\leq \beta(1)	\int_{\Lambda_1} (1+\abs{U(x)}+\abs{W(x)})^p\,dx\\
	&< \beta(1)	\norm{U-W}_{L^p(\Lambda;\RR^{M\times N})}^{\frac{p}{2}},
\end{align*}
since $(1+\abs{U(x)}+\abs{W(x)})^p<\norm{U-W}_{L^p(\Lambda;\RR^{M\times N})}^{-\frac{p}{2}}\abs{U(x)-W(x)}^p$ for every $x\in \Lambda_1$.
In addition, \eqref{vQuc} and the definition of $\Lambda_2$ immediately yield that
\begin{align*}
	&\int_{\Lambda_2} \abs{v(U(x))-v(W(x))}\,dx\\
	&\leq \beta\Big(\norm{U-W}_{L^p(\Lambda;\RR^{M\times N})}^{\frac{1}{2}}\Big)	\int_{\Lambda_2} (1+\abs{U(x)}^p+\abs{W(x)}^p)\,dx.
\end{align*}
Combining, we obtain \eqref{vLpL1uc} with $\gamma(\delta):=\beta(\delta^{\frac{1}{2}})+\beta(1)\delta^{\frac{p}{2}}$.
\end{proof}
We now recall some results of \cite{Kroe10d} on weak lower semicontinuity along purely concentrating sequences:
\begin{thm}\label{thm:lscconc}
Let $\O\subset\RR^N$ be open and bounded with boundary of class $C^1$, let $1<p<\infty$, let $\varphi\in C(\overline{\O})$ with $\varphi\geq 0$ on $\overline{\O}$, let $u\in W^{1,p}(\O;\RR^M)$ and let $v\in \ups$ satisfy \eqref{vLpL1uc}. If $v$ is $p$-qscb at every $x\in \overline{\O}$ with $\varphi(x)>0$, then
\begin{equation*}%\label{vlscconc}
	\liminf_{n\to\infty} \int_\O v(\nabla w_n(x)+\nabla u(x))\varphi(x)\,dx
	\geq \int_\O v(\nabla u(x))\varphi(x)\,dx
\end{equation*}
for every sequence $(w_n)\subset W^{1,p}(\O;\RR^M)$ which is bounded in $W^{1,p}$ and satisfies
$\abs{\{w_n\neq 0\}\cup \{\nabla w_n\neq 0\}}\to 0$.
\end{thm}
\begin{proof}
{\bf \underline{Step 1}: $\mbf{u=0}$.}\\
If $u=0$ and $v$ satisfies a $p$-Lipschitz condition, the assertion immediately follows from Theorem 3.5 and  Proposition 3.7 in \cite{Kroe10d}.
A closer look at the proofs of these results reveals that the $p$-Lipschitz condition is only used to show that $v:L^p\to L^1$ is uniformly continuous on bounded sets
(cf.~Proposition 2.4 in \cite{Kroe10d}), which we assumed in the form of \eqref{vLpL1uc}. (In fact, in \cite{Kroe10d}, the uniform continuity is exclusively used for arguments in the spirit of step 2 below.)

\noindent{\bf \underline{Step 2}: The general case.}\\
Clearly, $z_n:=\chi_{\{\nabla w_n\neq 0\}}\nabla u\to 0$ in $L^p(\O;\RR^{M\times N})$, and 
\begin{align*}
	&\int_\O v(\nabla w_n(x)+\nabla u(x))\varphi(x)\,dx
	-\int_\O v(\nabla u(x))\varphi(x)\,dx\\
	&=\int_\O v(\nabla w_n(x)+z_n(x))\varphi(x)\,dx
	-\int_\O v(z_n(x))\varphi(x)\,dx
\end{align*}
for every $n$. Hence, the general case reduces to the case for $u=0$ as a consequence of \eqref{vLpL1uc}.
\end{proof}
\begin{prop} \label{prop:lscconcnec}
Let $\O\subset\RR^N$ be open and bounded with boundary of class $C^1$, let $1<p<\infty$, let $v\in \ups$ satisfy \eqref{vLpL1uc}, let
$\bar{x}\in \overline\O$ and define $E:=B_1(0)$ if $\bar{x}\in\O$ and $E:=D_\nu$ if $\bar{x}\in\partial\O$, where
$\nu=\nu(\bar{x})$ is the outer normal to $\partial \O$ at $\bar{x}$ and $D_\nu:=\{y\in B_1(0)\mid y\cdot \nu<0\}$. 
If 
\begin{equation*}%\label{vlscconc}
	\liminf_{n\to\infty} \int_E v(\nabla w_n(y))\,dy
	\geq \int_E v(0)\,dy,
\end{equation*}
for every bounded sequence $(w_n)\subset W^{1,p}(B_1;\RR^M)$ such that
$w_n\to 0$ in $L^p$ and $\{w_n\neq 0\}\cup \{\nabla w_n\neq 0\}\subset B_{\frac{1}{n}}(0)$ for every $n$, then
$v$ is $p$-qscb at $\bar{x}$.
\end{prop}
\begin{proof}
If $v$ satisfies a $p$-Lipschitz condition, the assertion follows from Proposition 3.8 in \cite{Kroe10d} applied with $\O:=E=E(\bar{x})$ and $x_0:=0$. 
Moreover, as remarked before, the $p$-Lipschitz condition can be replaced by \eqref{vLpL1uc}.
\end{proof}
A closer look at the dependence of the definition of $p$-qscb at a point $x\in\partial\O$ on the outer normal $\nu(x)$ to $\partial\O$ at this point
reveals the following:
\begin{lem}\label{lem:pqscbrotated}
Let $1<p<\infty$, let $f:\RR^{M\times N}\to \RR$ be continuous, and let $\nu_{1},\nu_2\in S^{N-1}$.  	
If $R_{21}\in \RR^{N\times N}$ is an orthogonal matrix such that $\nu_2=R_{21}\nu_1$, then
$$
	\text{$s\mapsto f(s)$ is $\nu_1$-$p$-bqscb \quad if and only if \quad $s\mapsto f(s R_{21})$ is $\nu_2$-$p$-bqscb}. 
$$	
\end{lem}
\begin{proof}
Let $\varphi_1\in W_0^{1,p}(B_1;\RR^M)$. Using the notation of Definition~\ref{def:qscb}, 
we have that
$$
	\int_{D_{\nu_1}} f(\nabla \varphi_1)\,dx\geq -\eps \int_{D_{\nu_1}} \abs{\nabla \varphi_1}^p\,dx-C_\eps
$$
if and only if for $\varphi_2\in W_0^{1,p}(B_1;\RR^M)$, $\varphi_2(y):=\varphi_1\big(R_{21}^{-1} y\big)$,
$$
	\int_{D_{\nu_2}} f\big((\nabla \varphi_2) R_{21}\big)\,dy\geq -\eps \int_{D_{\nu_2}} \abs{\nabla \varphi_2}^p\,dy-C_\eps,
$$
by the change of variables given by $y=R_{21}x$. Here, note that $D_{\nu_2}=R_{21}D_{\nu_1}$, $\abs{\det R_{21}}=1$ and 
$\abs{(\nabla \varphi_2) R_{21}}=\abs{\nabla \varphi_2}$.
\end{proof}
\begin{prop}\label{prop:qscbRseparable}
Let $1\leq p<\infty$, assume that \eqref{H3} holds, and let $\nu\in S^{N-1}$. 
Then
$$
	G_\nu:=\mysetr{v_0\in \cR}{\text{$v$ is $\nu$-$p$-bqscb, where $v(s):=v_0(s)(1+\abs{s}^p)$}}.
$$
is the closure of its interior in $\cR$ (with respect to the trace topology of $L^\infty(\RR^{M\times N})$).
In particular, if $\cR_0$ is a dense subset of $\cR$, then $\cR_0\cap G_\nu$ is dense in $G_\nu$.
\end{prop}
\begin{proof}
For every $\delta>0$ and $v_0\in G$, $\delta+v_0(\cdot)$ is an interior point of $G$ (relative to $\cR$). 
\end{proof}

\section{Separating boundary and interior}

By means of a result of \cite{Kroe10d}, any bounded sequence in $W^{1,p}$ (up to a subsequence) can be split into a sum of two parts, 
the first ``purely concentrating'' at the boundary of the domain, while the second part does not charge the boundary in the sense made precise below.
This splitting has an analogon for DiPerna-Majda measures, decomposing $(\sigma,\hat{\nu})\in \DM$ 
into two parts $(\sigma_b,\hat{\nu}_b)$ and $(\sigma_i,\hat{\nu}_i)$ associated to the boundary and the interior of $\O$, respectively, as follows:
\begin{equation}\label{DMsplit}
\begin{alignedat}{2}
	&\sigma_b(dx)&&:=\chi_{\partial \O}(x)\sigma(dx)+dx,~~\\
	&\hat{\nu}_{b,x}(ds)&&:=\chi_{\partial\O}(x)\hat{\nu}_{x}(ds)+\chi_{\O}(x)\delta_{0}(ds),\\
	&\sigma_i(dx)&&:=\chi_{\O}(x)\sigma(dx),~~\\
	&\hat{\nu}_{i,x}(ds)&&:=\chi_{\O}(x)\hat{\nu}_{x}(ds),
\end{alignedat}
\end{equation}
where $\delta_0$ denotes the Dirac mass at $0\in \com$.
Assuming that $\abs{\partial\O}=0$, we have in particular that
\begin{equation}\label{DMsplit2}
\begin{aligned}
	&\int_{\overline{\O}} \varphi(x) \bscp{\hat{\nu}_{x}}{v_0}\,\sigma(dx)\\
	&\quad =\begin{aligned}[t]
		&\int_{\partial \O} \varphi(x) \bscp{\hat{\nu}_{b,x}}{v_0}\,\sigma_b(dx)
		+\int_{\O} \varphi(x) \bscp{\hat{\nu}_{i,x}}{v_0}\,\sigma_i(dx)
	\end{aligned}
\end{aligned}
\end{equation}
for every $\varphi\in C(\overline{\O})$ and every $v_0\in\cR$.
%$$
%	U(x):=d_\sigma(x)\int_{\RR^{M\times N}} \frac{s}{1+\abs{s}^p}\,\hat{\nu}_x(ds).
%$$
%Note that $(\sigma^U,\hat{\nu}^U)\in \DM$ given by $\sigma^U(dx)=(1+\abs{U(x)}^p)dx$ and $\hat{\nu}^U_x(ds)=\delta_{U(x)}(ds)$ is 
%the DiPerna-Majda measure generated by the constant sequence $(U(x))$.

The decomposition \eqref{DMsplit} does not affect the properties we are interested in:
\begin{prop}\label{prop:GDMsplit}
Let $\O\subset \RR^N$ be open and bounded with boundary of class $C^1$, let $1<p<\infty$ and let $(\sigma,\hat{\nu})\in \DM$.
Then $(\sigma_b,\hat{\nu}_b)\in \DM$ and $(\sigma_i,\hat{\nu}_i)\in \DM$. Moreover, the following assertions hold:
\begin{enumerate}
\item[(a)] $(\sigma,\hat{\nu})\in \GDM$ if and only if\\ 
both $(\sigma_b,\hat{\nu}_b)\in \GDM$ and $(\sigma_i,\hat{\nu}_i)\in \GDM$.
\item[(b)] If $(\sigma,\hat{\nu})\in \GDM$, then there exists $u\in W^{1,p}(\O;\RR^M)$
and bounded sequences $(u_{b,n}),(u_{i,n})\subset W^{1,p}(\O;\RR^M)$
such that
$$
	\begin{aligned}
		&\nabla u(x)=d_{\sigma_i}(x)\int_\com \frac{s}{1+\abs{s}^p}\,\hat{\nu}_{i,x}(ds),\\
		&\text{$u_{b,n}\rightharpoonup 0$ and $u_{i,n}\rightharpoonup u$ weakly in $W^{1,p}(\O;\RR^M)$,}\\
		&\{u_{b,n}\neq 0\}\subset (\partial\O)_{\frac{1}{n}}~~\text{and}~~\{u_{i,n}\neq u\}\subset \O\setminus (\partial\O)_{\frac{1}{n}},\\
		&\text{$(\nabla u_{b,n})$ generates $(\sigma_b,\hat{\nu}_b)$, $(\nabla u_{i,n})$ generates $(\sigma_i,\hat{\nu}_i)$}\\
		&\text{and $(\nabla u_{b,n}+\nabla u_{i,n})$ generates $(\sigma,\hat{\nu})$.}
	\end{aligned}
$$
\item[(c)] $(\sigma,\hat{\nu})$ satisfies (i)-(iii) in Theorem~\ref{thm:GDMchar} if and only if\\
$(\sigma_i,\hat{\nu}_i)$ satisfies (i)-(iii).
\item[(d)] $(\sigma,\hat{\nu})$ satisfies (iv) in Theorem~\ref{thm:GDMchar} if and only if\\
$(\sigma_b,\hat{\nu}_b)$ satisfies (iv).
\end{enumerate}
\end{prop}
The proof is given at the end of this section. Proposition~\ref{prop:GDMsplit} allows us to focus on the discussion of the boundary in the proof of our main result,
because the results of \cite{KaKru08a} can be applied to $(\sigma_i,\hat{\nu}_i)$ in a straightforward way. 

For the proof of (a), we first recall some results of \cite{Kroe10d} involving the following notion:
\begin{defn}\label{def:charge} 
Given a sequence $(u_n)\subset W^{1,p}(\O;\RR^M)$ and a closed set $K\subset \overline{\O}$, we say that \emph{$u_n$ does not charge $K$} (in $W^{1,p}$), if 
$$
	\sup_{n\in\NN}\int_{(K)_\delta\cap \O} \big( 
	 \abs{u_n}^p+\abs{\nabla u_n}^p \big)dx
	\underset{\delta\to 0^+}{\To} 0.
$$
Here, $(K)_{\delta}:=\bigcup_{x\in K}B_\delta(x)$ denotes the open $\delta$-neighborhood of $K$ in $\RR^N$.
\end{defn}
\begin{lem}[local decomposition in $W^{1,p}$, cf.~Lemma 2.6 in \cite{Kroe10d}]
%Lemma 2.6 in \cite{Kroe10d}
\label{lem:decloc}
Let $\O\subset \RR^N$ be open and bounded, let $1\leq p<\infty$
and let $K_j\subset \overline{\O}$, $j=1,\ldots,J$, be a finite family of compact sets such that $\overline{\O}\subset \bigcup_j K_j$.
Then for every bounded sequence $(u_n)\subset W^{1,p}(\O;\RR^M)$ 
with
$u_n\to 0$ in $L^p$,
there exists a subsequence $u_{k(n)}$ which can be decomposed as
$$
	u_{k(n)}=u_{1,n}+\ldots+u_{J,n},
$$
where for each  $j\in\{1,\ldots, J\}$, $(u_{j,n})_n$ is a bounded sequence in $W^{1,p}(\O;\RR^M)$ converging to zero in $L^p$ such that
the following three conditions hold:
\begin{enumerate}
	\item[(i)]
	$\{u_{j,n}\neq 0\}\subset \{u_n\neq 0\}$,
	$\{\nabla u_{j,n}\neq 0\}\subset \{\nabla u_n\neq 0\}$ 
	(possibly ignoring a set of measure zero)
	and $\overline{\{u_{j,n}\neq 0\}}
	\subset (K_j)_{\frac{1}{n}}
	\setminus \textstyle{\bigcup_{i<j}}K_i$
	for every $j,n$,
	\item[(ii)]
	$u_{j,n}$ does not charge 
	$\textstyle{\bigcup_{i<j}}K_i$ in $W^{1,p}$ for each $j$.
	\item[(iii)] On the ``transition layer''
	$$
		T_n:=\mysetr{x\in\O}{u_{j,n}(x)\neq 0~\text{for at least two different $j$}},
	$$
	we have that
	$$
		\int_{T_n} \big(\abs{u_{j,n}}^p+\abs{\nabla u_{j,n}}^p\big)\,dx
		\underset{n\to\infty}{\To}0,~~\text{for $j=1,\ldots,J$.}
	$$
\end{enumerate}
%Moreover, if $(u_n)$ is $W^{1,p}$-equiintegrable, then so is $(u_{j,n})_n$ for each $j$, and 
%Moreover, if $(u_n)\subset W_0^{1,p}(\O;\RR^M)$ then also $(u_{j,n})_n\subset W_0^{1,p}(\O;\RR^M)$. 
Here, $(K_j)_{\frac{1}{n}}$ denotes the open $\frac{1}{n}$-neighborhood of $K_j$ in $\RR^N$ as before.
\end{lem}%
For our purposes here, the case $J=2$, $K_1=\partial \O$ and $K_2=\O$ in Lemma~\ref{lem:decloc} suffices.
\begin{proof}[Proof of Lemma~\ref{lem:decloc}]
See Lemma 2.6 in \cite{Kroe10d}. Condition (iii) is not stated in \cite{Kroe10d}, but it is an immediate consequence of the proof provided there.
\end{proof}
Because of (iii), the component sequences above essentially do not interact, and we 
are able to split nonlinear expressions as well, cf.~Proposition 2.7 in \cite{Kroe10d}:
\begin{prop}\label{prop:fsplit}
Let $\O\subset \RR^N$ be open and bounded and let $1\leq p<\infty$. In addition, assume that
$f:\RR^{M\times N}\to \RR$ is continuous and satisfies a $p$-growth condition (i.e.,
$s\mapsto (1+\abs{s}^p)^{-1}f(s)$ is bounded).
Then for every $U\in L^p(\O;\RR^{M\times N})$,
$$
	f(\nabla u_n+U)-f(U)-\sum_{j=1}^J \big(f(\nabla u_{j,n}+U)-f(U)\big)
	\underset{n\to\infty}{\To} 0\quad\text{in $L^1(\O)$},
$$
for any decomposition $u_n=\sum_j u_{j,n}$ into a finite sum of bounded sequences in $W^{1,p}(\O;\RR^M)$ such that Lemma~\ref{lem:decloc} (iii) holds.
\end{prop}
\begin{proof}
Observe that since $u_n=\sum_j u_{j,n}$, 
the definition of the set $T_n$ in condition (iii) of 
Lemma~\ref{lem:decloc} yields that
$$
  f(\nabla u_n+U)-f(U)-\sum_{j=1}^J \big(f(\nabla u_{j,n}+U)-f(U)\big)=0
  \quad\text{a.e.~on $\O\setminus T_n$}.
$$
Hence, it suffices to show that
$f(\chi_{T_n}\nabla u_n+U)\to f(U)$ and $f(\chi_{T_n}\nabla u_{j,n}+U)\to f(U)$ in $L^1(\O)$, for $j=1,\ldots, J$. This is a consequence of (iii), since
our assumptions on $f$ imply that $V\mapsto f(V)$, $L^p(\O;\R^{M\times N})\to L^1(\O)$, is continuous. 
\end{proof}
\begin{proof}[Proof of Proposition~\ref{prop:GDMsplit}]
Using Proposition~\ref{prop:DMchar}, it is not difficult to check that $(\sigma_b,\hat{\nu}_b),(\sigma_i,\hat{\nu}_i)\in \DM$,
and both (c) and (d) readily follow from \eqref{DMsplit}. It remains to show (a) and (b).

{\bf (a) ``only if'':} Suppose that $(\sigma,\hat{\nu})$ is generated by $(\nabla u_n)$, for a bounded sequence $(u_n)\subset W^{1,p}(\O;\RR^M)$
such that $u_n\rightharpoonup u$ weakly in $W^{1,p}$ for some $u\in W^{1,p}(\O;\RR^M)$. By the compact embedding of $W^{1,p}$ into $L^p$,
we also have that $u_n\to u$ strongly in $L^p$. We decompose (up to a subsequence, not relabeled)
$$
	u_n-u=u_{1,n}+u_{2,n}
$$
according to Lemma~\ref{lem:decloc}, applied with $J=2$, $K_1:=\partial\O$ and $K_2:=\overline{\O}$.
Let $(\sigma_1,\hat{\nu}_1)$ and $(\sigma_2,\hat{\nu}_2)$ denote the DiPerna-Majda measures generated by $(\nabla u_{1,n})$ and 
$(\nabla u_{2,n}+\nabla u)$, respectively (up to a subsequence).
By construction, $\{u_{1,n}\neq 0\}\cup \{\nabla u_{1,n}\neq 0\} \subset (\partial\O)_{\frac{1}{n}}$ for every $n$, and $(\nabla u_{2,n})$ does not charge $\partial\O$ in $L^p$.
This implies that 
\begin{equation}\label{pGDMs-1}
	\text{$\sigma_1(dx)=dx$ on $\O$,$\quad$$\hat{\nu}_{1,x}=\delta_0$ for a.e.~$x\in \O$$\quad$and$\quad$$\sigma_2(\partial\O)=0$.}
\end{equation}
Moreover, $\chi_{\{\nabla u_{1,n}\neq 0\}}\nabla u\to 0$ in $L^p$, whence
\begin{equation}\label{pGDMs-2}
	\big[v(\nabla u_{1,n}+\nabla u)-v(\nabla u)\big]
	-\big[v(\nabla u_{1,n})-v(0)\big]
	\underset{n\to\infty}{\To}0~~\text{in $L^1(\O)$}
\end{equation}
for every $v\in \ups$,
due to the uniform continuity of the Nemytskii operator associated to $v$ on bounded subsets of $L^p$, cf~\eqref{vLpL1uc}.
Proposition~\ref{prop:fsplit} applied with $f=v$ and $U=\nabla u$ additionally yields that
\begin{equation}\label{pGDMs-3}
	v(\nabla u_n)-\big[v(\nabla u_{1,n}+\nabla u)-v(\nabla u)]-v(\nabla u_{2,n}+\nabla u)
	\underset{n\to\infty}{\To}0~~\text{in $L^1(\O)$}.
\end{equation}
Combining \eqref{pGDMs-1}--\eqref{pGDMs-3}, we infer that
\begin{align*}
	&\int_{\overline{\O}} \varphi(x) \bscp{\hat{\nu}_{x}}{v_0}\,\sigma(dx)
	\\
	&=	\int_{\overline{\O}} \varphi(x) \bscp{\hat{\nu}_{1,x}}{v_0} \,\sigma_1(dx)-\int_\O v(0)\,dx
	+\int_{\overline{\O}} \varphi(x) \bscp{\hat{\nu}_{2,x}}{v_0} \,\sigma_2(dx)\\
	&=	\int_{\partial \O} \varphi(x) \bscp{\hat{\nu}_{1,x}}{v_0} \,\sigma_1(dx)
	+\int_{\O} \varphi(x) \bscp{\hat{\nu}_{2,x}}{v_0} \,\sigma_2(dx)
\end{align*}
for every $\varphi\in C(\O)$ and every $v_0\in \cR$, where $v(s):=v_0(s)(1+\abs{s}^p)$.
By comparison with \eqref{DMsplit2}, we get that $(\sigma_b,\hat{\nu}_b)=(\sigma_1,\hat{\nu}_1)\in \GDM$
and  $(\sigma_i,\hat{\nu}_i)=(\sigma_2,\hat{\nu}_2)\in \GDM$ as claimed.

{\bf (a) ``if'':} Suppose that $(\sigma_b,\hat{\nu}_b)$ is generated by $(\nabla w_{b,n})$ and $(\sigma_i,\hat{\nu}_i)$ is generated by $(\nabla w_{i,n})$, 
for some bounded sequences $(w_{b,n})_n,(w_{i,n})_n\subset W^{1,p}(\O;\RR^M)$. 
In particular,
$$
	\int_\O \varphi(x)\abs{\nabla w_{b,n}(x)}^p\,dx\to \int_{\overline{\O}}\int_\com \frac{\abs{s}^p}{1+\abs{s}^p}\,\hat{\nu}_{b,x}(ds)\sigma_b(dx)=0
$$
for every $\varphi\in C_0(\O)$, whence $\nabla w_{b,n}\to 0$ in $L^p_\loc(\O;\RR^{M\times N})$. Passing to a subsequence and adding a suitable constant to $w_{b,n}$ (if necessary; this does not change the gradient which is the only thing that matters for us), we also may assume that $w_{b,n}\to 0$ in $L^p$ by compact embedding.
In addition, 
$$
	\int_\O \varphi(x)\nabla w_{i,n}(x)\,dx\to \int_{\overline{\O}}\int_\com \frac{s}{1+\abs{s}^p}\,\hat{\nu}_{i,x}(ds)\sigma_i(dx),
$$
whence (up to a subsequence)
$$
	w_{i,n}\rightharpoonup u~~\text{weakly in $W^{1,p}(\O;\RR^M)$, where}~~\nabla u(x)=d_{\sigma_i}(x)\int_\com \frac{s}{1+\abs{s}^p}\,\hat{\nu}_{i,x}(ds).
$$
A natural choice for a generating sequence of $(\sigma,\hat{\nu})$ is $(\nabla u_n)$ with $u_n=w_{b,n}+w_{i,n}$; 
however, this only works well if the interaction of the two component sequences,
which in principle could occur on the set $\{w_{b,n}\neq 0\}\cap \{w_{i,n}\neq u\}$, is negligible. 
We thus first modify $w_{b,n}$ and $w_{i,n}$, in such a way that this set becomes empty. 

For this purpose, choose two sequences $(\varphi_n),(\eta_n)\subset C^1_c(\O;[0,1])$ such that 
$\varphi_n=1$ on $\O\setminus (\partial \O)_{\frac{1}{n}}$, $\eta_n=0$ on $(\partial \O)_{\frac{1}{n}}$
and $\eta_n=1$ on $\O\setminus (\partial \O)_{\frac{2}{n}}$ for every $n$.
For every fixed $n$, we have that $\varphi_n w_{b,k}\to 0$ in $W^{1,p}$ and 
$(\nabla (1-\eta_n))\otimes (w_{i,k}-u)\to 0$ in $L^p$ as $k\to\infty$.
Due to the latter, we also obtain that
\begin{align*}
	&\lim_{k\to\infty}\int_\O \abs{\nabla \big((1-\eta_n) (w_{i,k}-u)\big)}^p\,dx\\
	&=\lim_{k\to\infty}\int_\O \abs{1-\eta_n}^p\abs{\nabla w_{i,k}-\nabla u}^p\,dx\\
	&\leq \lim_{k\to\infty} 2^p\int_\O \abs{1-\eta_n}^p(1+\abs{\nabla w_{i,k}}^p)\,dx
	+2^p \int_\O \abs{1-\eta_n}^p \abs{\nabla u}^p\,dx\\
	&=2^p \int_{\overline{\O}} \abs{1-\eta_n(x)}^p 
	\hat{\nu}_{i,x}(\com)\,\sigma_i(dx)+2^p\int_\O \abs{1-\eta_n}^p \abs{\nabla u}^p\,dx,
\end{align*}
whence
\begin{align*}
	\lim_{n\to\infty}\lim_{k\to\infty}\int_\O \abs{\nabla \big((1-\eta_n) (w_{i,k}-u)\big)}^p\,dx
	\leq 2^p \sigma_i(\partial\O)=0
\end{align*}
by dominated convergence. As a consequence, there exists a subsequence
$k(n)$ of $n$ such that as $n\to\infty$,
\begin{equation}\label{pGDMs-11}
	\text{$\varphi_n w_{b,k(n)}\to 0$ in $W^{1,p}$}\quad\text{and}\quad
	\text{$(1-\eta_n) (w_{i,k(n)}-u)\to 0$ in $W^{1,p}$}
\end{equation}
We define
$$
	u_{b,n}:=(1-\varphi_n)\cdot w_{b,k(n)}
	\quad\text{and}\quad
	u_{i,n}:=\eta_n \cdot(w_{i,k(n)}-u)+u.
$$
Note that by \eqref{pGDMs-11} and \eqref{vLpL1uc},
$(\nabla u_{b,n})$ and $(\nabla u_{i,n})$ still generate $(\sigma_b,\hat{\nu}_b)$ and $(\sigma_i,\hat{\nu}_i)$, respectively.
Moreover, for
$$
	u_n:=u_{b,n}+u_{i,n},
$$
the decomposition $u_n-u=u_{b,n}+(u_{i,n}-u)$ is admissible in Proposition~\ref{prop:fsplit} 
(note that $\{u_{b,n}\neq 0\}\cap \{u_{i,n}-u\neq 0\}=\emptyset$ by construction),
and arguing as in the proof of (i) ``only if'', we obtain that 
\begin{align*}
	&\lim_{n\to\infty} \int_\O \varphi(x) v(\nabla u_n(x))\,dx
	=	\int_{\partial \O} \varphi(x) \bscp{\hat{\nu}_{b,x}}{v_0} \,\sigma_b(dx)
	+\int_{\O} \varphi(x) \bscp{\hat{\nu}_{i,x}}{v_0}\,\sigma_i(dx),
\end{align*}
for every $\varphi\in C(\overline{\O})$ and every $v_0\in \cR$, with $v(s):=v_0(s)(1+\abs{s}^p)$.
In view of \eqref{DMsplit2}, this means that $(\nabla u_n)$ generates $(\sigma,\hat\nu)$.

{\bf (b):} The function $u$ and the sequences $(u_{b,n})$ and $(u_{i,n})$ 
obtained in the previous step have all the asserted properties.
\end{proof}

\section{Necessary conditions\label{sec:nec}}

We now prove that each $(\sigma,\nu)\in \GDM$ satisfies the conditions (i)--(iv) of Theorem~\ref{thm:GDMchar}.
The conditions in the interior of $\O$ follow from the associated result of \cite{KaKru08a}:
\begin{thm}[cf.~Theorem 2.8 in \cite{KaKru08a}]
Assume that \eqref{H1}--\eqref{H3} hold, and let $(\sigma,\hat{\nu})\in \GDM$ be generated by $(\nabla u_n)$ such that 
$u_n \rightharpoonup u$ weakly in $W^{1,p}(\O;\RR^M)$. Then $(\sigma,\hat{\nu})$ satisfies (i)--(iii) in Theorem~\ref{thm:GDMchar}.
\end{thm}
\begin{remark}
In fact, Theorem 2.8 in \cite{KaKru08a} uses weaker assumptions: it suffices to have that $\abs{\partial\O}=0$ instead of a boundary of class $C^1$, 
and \eqref{v0ltuc} is not needed there.
\end{remark}
It remains to show (iv): 
\begin{prop}\label{prop:nec1}
Assume that \eqref{H1}--\eqref{H3} hold, and let $(\sigma,\hat{\nu})\in \GDM$.
Then $(\sigma,\hat{\nu})$ satisfies (iv) in Theorem~\ref{thm:GDMchar}, i.e.,
\begin{equation}\label{pnec1iv}
	\scp{\hat{\nu}_{x}}{\frac{v(\cdot)}{1+\abs{\cdot}^p}}_\infty= \int_{\com\setminus \RR^{M\times N}} \frac{v(s)}{1+\abs{s}^p} \,\hat{\nu}_{x}(ds)~\geq~0
\end{equation}
for $\sigma$-a.e.~$x_0\in \partial \O$ and every $v\in \ups$ which is $p$-qscb at $x_0$.
\end{prop}

\begin{proof}%[Proof of Proposition~\ref{prop:nec1}]
In view of \eqref{DMsplit}, it suffices to show that $(\sigma_b,\nu_b)$ satisfies 
\eqref{pnec1iv}.
By Proposition~\ref{prop:GDMsplit}, we have that $(\sigma_b,\nu_b)\in \GDM$.
Let $x_0\in \partial\O$, 
let $v\in \ups$ be $p$-qscb at $x_0$, 
and let $(u_n)\subset W^{1,p}(\O;\RR^M)$ be a bounded sequence such that 
$\{u_n\neq 0\}\subset (\partial \O)_{\frac{1}{n}}$ and$(\nabla u_n)$ generates $(\sigma_b,\hat{\nu}_b)$.
In particular, $u_n\rightharpoonup 0$ weakly in $W^{1,p}(\O;\RR^M)$.
For fixed $\eps>0$, due to \eqref{vLpL1uc}, $v_\eps(\cdot):=v(\cdot)-v(0)+\eps(1+\abs{\cdot}^p)$ is even $p$-qscb at every $x\in \cU_\eps\cap \bar{\O}$ for a neighborhood $\cU_\eps$ of $x_0$ in $\RR^N$. 
If $\varphi\in C(\overline{\O})$, $\varphi\geq 0$ and $\{\varphi\neq 0\}\subset \cU_\eps$, Theorem~\ref{thm:lscconc} yields that 
\begin{equation}\label{pnec1-2}
\begin{aligned}
	 0&\leq \lim_{n\to \infty} \int_\O \varphi(x)v_\eps(\nabla u_n(x))\,dx\\
	 &=\int_{\overline\O}\varphi(x)\scp{\hat{\nu}_{b,x}}{\frac{v_\eps(\cdot)}{1+\abs{\cdot}^p}}	\,\sigma_b(dx)\\
	 &=\int_{\partial\O}\varphi(x)
	 \scp{\hat{\nu}_{b,x}}{\frac{v_\eps(\cdot)}{1+\abs{\cdot}^p}}_\infty	\,\sigma_b(dx).
\end{aligned}
\end{equation}
The last equality in \eqref{pnec1-2} holds because $\hat{\nu}_{b,x}=\delta_0$ for $\sigma_b$-a.e.~$x\in\O$, $v_\eps(0)=0$ and 
$\hat{\nu}_{b,x}(\RR^{M\times N})=0$ for $\sigma_b$-a.e.~$x\in \partial\O$ by Remark~\ref{rem:DMprop}.
Since $\varphi$ is arbitrary with non-negative values on any $\cV_\eps$ compactly contained in $\cU_\eps$,
\eqref{pnec1-2} implies that
\begin{equation*}
	0\leq \int_{\partial\O\cap \overline{\cV}_\eps}
	\scp{\hat{\nu}_{b,x}}{\frac{v_\eps(\cdot)}{1+\abs{\cdot}^p}}_\infty\,\sigma_b(dx)
\end{equation*}
by dominated convergence. As a consequence, we have that
\begin{equation}\label{pnec1-3}
	-\eps\leq \frac{1}{\sigma(\partial\O\cap \overline{\cV}_\eps)} 
	\int_{\partial\O\cap \overline{\cV}_\eps}\scp{\hat{\nu}_{b,x}}{\frac{v(\cdot)}{1+\abs{\cdot}^p}}_\infty\,\sigma_b(dx)
\end{equation}
as long as $\sigma_b(\partial\O\cap \overline{\cV}_\eps)>0$.
In the limit $\eps\to 0^+$, we infer that 
\begin{equation}\label{pnec1-4}
	0\leq \scp{\hat{\nu}_{b,x_0}}{\frac{v(\cdot)}{1+\abs{\cdot}^p}}_\infty,
\end{equation}
provided that $x_0$ is a $\sigma_b$-Lebesque point of the right hand side of \eqref{pnec1-4},
i.e., for $\partial\O\to \RR$, $x\mapsto \bscp{\hat{\nu}_{b,x}}{v_0}_\infty$ with $v_0:=v(\cdot)(1+\abs{\cdot}^p)^{-1}$.
Now choose a countable subset $\cR_0$ which is dense in $\cR$. There exists a set $Z\subset \partial\O$ such that $\sigma_b(Z)=0$ and for every $v_0\in \cR_0$,
$\partial\O\setminus Z$ is a subset of the $\sigma_b$-Lebesgue points of $x\mapsto \bscp{\hat{\nu}_{b,x}}{v_0}_\infty$. In particular, \eqref{pnec1-4} holds for every $x_0\in \partial\O\setminus Z$ and every 
$v\in \ups$ such that $v$ is $p$-qscb at $x$ and $v(\cdot)(1+\abs{\cdot}^p)^{-1}\in \cR_0$.
By density, also using Proposition~\ref{prop:qscbRseparable}, this implies the assertion.
\end{proof}

\section{Sufficient conditions\label{sec:suff}}

By Proposition \ref{prop:GDMsplit}, $(\sigma,\hat{\nu})\in \GDM$ provided that 
$(\sigma_b,\hat{\nu}_b)\in \GDM$ and $(\sigma_i,\hat{\nu}_i)\in \GDM$.
For this reason, the interior part and the boundary part can be studied separately.

\subsection{Sufficient conditions in the interior}~\\
As in the case of necessary conditions, we rely on a corresponding result of \cite{KaKru08a}, which, 
besides the conditions we stated as (i)--(iii) in Theorem~\ref{thm:GDMchar}, 
also uses the following condition for $(\sigma,\hat{\nu})\in \DM$ on the boundary, which is slightly stronger than (iv):
\begin{equation} \label{KaKru-boundary}
	\begin{aligned}
	&0\leq \scp{\hat{\nu}_{x}}{\frac{v(\cdot)}{1+\abs{\cdot}^p}}_\infty = \int_{\com\setminus \RR^{M\times N}} \frac{v(s)}{1+\abs{s}^p}\hat{\nu}_x(ds)\\
	&\text{for $\sigma$-a.e.~$x\in\partial\O$ and every $v\in\ups$ with $Qv>-\infty$.}
	\end{aligned}
\end{equation}
\begin{thm}[cf.~Theorem 2.7 in \cite{KaKru08a}]\label{thm:KaKrusuff}
Assume that \eqref{H1}--\eqref{H3} hold and let $(\sigma,\hat{\nu})\in \DM$. Then
\begin{enumerate}
\item[] there exists a bounded sequence $(u_n)\subset W^{1,p}(\O;\RR^M)$ with fixed boundary values\footnotemark 
such that $(\nabla u_n)$ generates $(\sigma,\hat{\nu})$
\end{enumerate}
\footnotetext{i.e., $u_n=u_m$ on $\partial\O$ in the sense of trace for every $n,m\in\NN$}
if and only if
\begin{enumerate}
\item[] (i)--(iii) in Theorem~\ref{thm:GDMchar} and  \eqref{KaKru-boundary} hold.
\end{enumerate}
\end{thm}
\begin{remark}
For Theorem 2.7~in \cite{KaKru08a}, it suffices to have a bounded domain with the extension property in $W^{1,p}$ (instead of $C^1$-boundary), 
and our assumption \eqref{v0ltuc} is not needed in \cite{KaKru08a}, either.
\end{remark}
In particular, Theorem~\ref{thm:KaKrusuff} tells us in which cases the interior part $(\sigma_i,\hat{\nu}_i)$ of $(\sigma,\hat{\nu})$, as defined in \eqref{DMsplit},
is generated by gradients:
\begin{cor}\label{cor:KaKrusuff}
Assume that \eqref{H1}--\eqref{H3} hold. If $(\sigma,\hat{\nu})\in \DM$ satisfies (i)--(iii) in Theorem~\ref{thm:GDMchar},
then $(\sigma_i,\hat{\nu}_i)\in \GDM$.
\end{cor}
\begin{proof}
By Proposition~\ref{prop:GDMsplit}, $(\sigma_i,\hat{\nu}_i)$ satisfies (i)--(iii), 
and \eqref{KaKru-boundary} trivially holds for $(\sigma_i,\hat{\nu}_i)$ since $\sigma_i(\partial \O)=0$.
Theorem~\ref{thm:KaKrusuff} thus yields the assertion.
\end{proof}

\subsection{Sufficient conditions at the boundary}~\\
Recall that condition (iv) in Theorem~\ref{thm:GDMchar} states
that
\begin{equation} \label{BGDM}
	\begin{aligned}
	&0\leq \scp{\hat{\nu}}{\frac{v(\cdot)}{1+\abs{\cdot}^p}}_\infty = \int_{\com\setminus \RR^{M\times N}} \frac{v(s)}{1+\abs{s}^p}\hat{\nu}_x(ds)\\
	&\text{for $\sigma$-a.e.~$x\in\partial\O$ and every $v\in\ups$ which is $p$-qscb at $x$.}
	\end{aligned}
\end{equation}
Below, the set of all DiPerna-Majda measures with this property (``boundary gradient DiPerna-Majda measures'') is denoted by
$$
	\BGDM:=\mysetl{(\sigma,\hat{\nu})\in \DM}{(\sigma,\hat{\nu})~\text{satisfies~\eqref{BGDM}}}.
$$
In two steps, we now prove for each $(\sigma,\hat{\nu})\in \BGDM$, its boundary part
$(\sigma_b,\hat{\nu}_b)$ as defined in \eqref{DMsplit} is generated by a sequence of gradients, throughout assuming that \eqref{H1}--\eqref{H3} hold.
\begin{thm}\label{thm:suffboundary}
%Assume that \eqref{H1}--\eqref{H3} hold, 
Let $(\sigma,\hat{\nu})\in \DM$ 
and suppose that $(\sigma,\hat{\nu})\in \BGDM$, i.e., $(\sigma,\hat{\nu})$ satisfies (iv) in Theorem~\ref{thm:GDMchar}.
Then $(\sigma_b,\hat{\nu}_b)\in \GDM$.
\end{thm}

\subsection*{Step 1: Measures supported on a single point on the boundary}~\\
If $\sigma$ charges a single boundary point, i.e., $\sigma(\partial\O\setminus\{x\})=0$ for some $x\in\partial\O$,
it suffices to study $\sigma(\{x\})\hat{\nu}_x$ instead of $(\sigma,\hat{\nu})$ on $\partial\O$.
Moreover, only the behavior on $\com\setminus \RR^{M\times N}$ matters since $\sigma(\{x\})\hat{\nu}_x(\RR^{M\times N})=0$ 
by Remark~\ref{rem:DMprop}.
For $x\in\partial \O$, we define two sets of measures of this kind: 
\begin{equation*}
	A_x:=\mysetr{\hat{\mu}\in \rca(\com)}
	{\begin{array}{l}\hat{\mu}\geq 0,~~\hat{\mu}(\RR^{M\times N})=0~~\text{and}~~\bscp{\hat{\mu}}{v_0}_\infty\geq 0~\\
	\text{for every}~\text{$v\in\ups$ which is $p$-qscb at $x$}\end{array}},
\end{equation*}
where $v_0:=\frac{v(\cdot)}{1+\abs{\cdot}^p}$ as usual.
The second set $H_x$, defined below, consists of measures generated by certain ``purely concentrating'' sequences: 
\begin{defn}\label{def:pointconc}
Let $x\in\overline{\O}$. We say that $\hat{\delta}=\hat{\delta}_{x,(\nabla u_n)}\in \rca(\com)$ is a \emph{gradient point concentration measure at $x$}
if there exists a bounded sequence $(u_n)\subset W^{1,p}(\O;\RR^M)$ such that the following two properties hold:
\begin{enumerate}
\item[(a)]
$\{u_n\neq 0\}\subset B_{r_n}(x)$ for some sequence $r_n\to 0^+$,
and 
\item[(b)] for every $v\in \ups$, the limit below exists and satisfies
\begin{equation*}%\label{pcmeasO}
	\quad\bscp{\hat{\delta}}{v_0}= \int_{\com} v_0(s)\, \hat{\delta}(ds)=\lim_{n\to \infty}
	\int_{\O} v(\nabla u_{n}(y))\,dy-\abs{\O} v(0),
\end{equation*}
where $v_0:=\frac{v(\cdot)}{1+\abs{\cdot}^p}$. 
\end{enumerate}
In this case, we say that \emph{$\hat{\delta}$ is generated by $(\nabla u_n)$}.
\end{defn}%
For $x\in\partial \O$ we now set
\begin{equation*}
	H_x:=\mysetr{\hat{\delta}\in\rca(\com)}
	{\hat{\delta}~\text{is a gradient point concentration measure at $x$}}.
\end{equation*}
In the present context, the desired sufficient condition amounts to proving that $A_x\subset H_x$. The proof is carried out in a series of propositions, the first of which provides an equivalent formulation of Definition~\ref{def:pointconc} which is technically more convenient for us. 
\begin{prop}\label{prop:pointconcOtoD}
Let $x\in \partial \O$, let $D=D(x):=\{y\in B_1(0)\mid y\cdot \nu(x)<0\}$, where $\nu(x)$ is the outer normal to $\partial \O$ at $x\in \partial \O$
and let $\hat{\delta}\in \rca(\com)$. Then $\hat{\delta}$ is a gradient point concentration measure at $x$ if and only if
if there exists a bounded sequence $(\tilde{u}_n)\subset W^{1,p}(D;\RR^M)$ with 
the following two properties:
\begin{enumerate}
\item[(a)] $\{\tilde{u}_n\neq 0\}\subset B_{r_n}(0)$ for some sequence $r_n\to 0^+$,
and
\item[(b)] for every $v\in \ups$, the limit below exists and 
\begin{equation*}%\label{pcmeasD}
	\quad\bscp{\hat{\delta}}{v_0}= \int_{\com} v_0(s)\, \hat{\delta}(ds)=\lim_{n\to \infty}
	\int_{D} v(\nabla \tilde{u}_{n}(y))\,dy-\abs{D} v(0),
\end{equation*}
where $v_0:=\frac{v(\cdot)}{1+\abs{\cdot}^p}$.
\end{enumerate}
\end{prop}
\begin{proof}
Since $\partial\O$ is of class $C^1$, there exists a $C^1$-diffeomorphism $\Phi$ mapping  
a neighborhood $\cV\subset B_1(0)$ of the origin onto
a neighborhood $\cU$ of $x$ in $\RR^N$ such that $\Phi(0)=x$, $D\Phi(0)=I$,
$\Phi(\cV\cap D)=\cU\cap \O$, and $\Phi(\{y\in \cV \mid y\cdot \nu(x)=0\})=\cU\cap \partial\O$.
If $(u_n)\subset W^{1,p}(\O;\RR^M)$ is a bounded sequence with support shrinking to $x$ such that
$(\nabla u_n)$ generates $\hat{\delta}$ in the sense of Definition~\ref{def:pointconc} (b),
then
$$
	\tilde{u}_n(z):=u(\Phi(z)),\quad z\in \cV,
$$ 
defines a bounded sequence $(\tilde{u}_n)\subset W^{1,p}(D;\RR^M)$ with support shrinking to the origin.
We claim that $(\nabla \tilde{u}_n)$ generates $\hat{\delta}$ in the sense of (b) above:
For any $v\in\ups$ and $\tilde{v}:=v(\cdot)-v(0)$, a change of variables yields that
$$
	\int_D \tilde{v}(\nabla \tilde{u}_n(z))\,dz
	%=\int_D \tilde{v}\big((\nabla u_n)(\Phi(z))D\Phi(z)\big)\,dz
	=\int_\O \tilde{v}\big(\nabla u_n(y)D\Phi(\Phi^{-1}(y))\big) \abs{\det D(\Phi^{-1})(y)}\,dy,
$$
and since $D\Phi(\Phi^{-1}(y))\to I$ and $D(\Phi^{-1})(y)\to I$ as $y\to x$ (recall that the support of $u_n$ shrinks to $x$ as $n\to\infty$),
\eqref{vLpL1uc} implies that
$$
	\lim_{n\to\infty}\int_\O \tilde{v}\big(\nabla u_n(y)D\Phi(\Phi^{-1}(y))\big) \abs{\det D(\Phi^{-1})(y)}\,dy
	=\lim_{n\to\infty}\int_\O \tilde{v}(\nabla u_n(y))) \,dy.
$$
As a consequence, we get that
$$
	\lim_{n\to\infty} \int_D v(\nabla \tilde{u}_n(z))\,dz-\abs{D}v(0)=\lim_{n\to\infty}\int_\O v(\nabla u_n(y))) \,dy-\abs{\O}v(0),
$$
and Definition~\ref{def:pointconc} (b) implies (b) as stated in the assertion.
Analogously, we can define $(u_n)$ starting from $(\tilde{u}_n)$ without changing the measure that is generated by the respective gradients.
\end{proof}
%\begin{prop}\label{prop:HxsubsetAx}
%For every $x\in\partial\O$, $H_x\subset A_x$.
%\end{prop}
%\begin{proof}
%Let $\hat{\delta}\in H_x$, generated by $(\nabla u_n)$ where $(u_n)\subset W^{1,p}(\O;\RR^M)$ is a bounded sequence with support shrinking to $x$.
%Up to a subsequence, $(\nabla u_n)$ also generates a DiPerna-Majda measure $(\sigma,\hat{\nu})\in \DM$. 
%Since the support of $(\nabla u_n)$ is also shrinking to $x$, we have that $\sigma(dy)=dy$ in $\O$, $\hat{\nu}_y(ds)=\delta_0(ds)$ for $y\in\O$, and %$\sigma(\partial\O\setminus\{x\})=0$.
%Definition~\ref{def:pointconc} (b) thus implies that 
%$\hat{\delta}=\sigma(\{x\})\hat{\nu}_x$.
%In particular, $\hat{\delta}\geq 0$, and $\hat{\delta}(\RR^{M\times N})=0$ by 
%Remark~\ref{rem:DMprop}. In addition,
%for every $v\in \ups$ which is $p$-qscb at $x$, with $v_0:=\frac{v(\cdot)}{1+\abs{\cdot}^p}$, we have that
%$$
%	\bscp{\hat{\delta}}{v_0}=\lim_{n\to\infty} \int_\O v(\nabla u_n)\,dx-\abs{\O}v(0)\geq 0
%$$ 
%by Theorem~\ref{thm:lscconc}.
%\end{proof}
\begin{prop}\label{prop:Hxconvex}
For every $x\in\partial\O$, $H_x$ is convex.
\end{prop}
\begin{proof}
Let $\hat{\delta}_1$ and $\hat{\delta}_2$ be two point concentrations at $x$, and let $\lambda\in (0,1)$. 
By Proposition~\ref{prop:pointconcOtoD}, $\hat{\delta}_1$ and $\hat{\delta}_2$, respectively, are generated  
by $(\nabla u_n)$ and $(\nabla w_n)$,
where $(u_n)$ and $(w_n)$ are suitable bounded sequences in $W^{1,p}(D;\RR^M)$ with support shrinking to the origin.
With a fixed unit vector $e$ tangential to $\partial\O$ at $x\in \partial\O$ (perpendicular to $\nu(x)$), 
we define
$$
	q_n(y):=\lam^{\frac{1}{N}} u_n\Big(\lam^{-\frac{1}{N}}y+r_n e\Big)+(1-\lam)^{\frac{1}{N}} w_n\Big((1-\lam)^{-\frac{1}{N}}y-r_n e\Big),\quad y\in D.
$$
Note that two summands of $q_n$ have disjoint support, and the support of $q_n$ is also shrinking to the origin as $n\to\infty$.
For every $v\in \ups$ and $\tilde{v}:=v(\cdot)-v(0)$,
a change of variables yields that
$$
	\int_{D} \tilde{v}(\nabla q_n(y))\,dy=\lam \int_D \tilde{v}(\nabla u_n(z))\,dz+(1-\lam) \int_D \tilde{v}(\nabla w_n(z))\,dz
$$
for every $n$ large enough so that the support of $q_n$ is contained in $\overline{D}$. Thus, 
$$
	\lim_{n\to\infty}\int_{D} v(\nabla q_n(y))\,dy -\abs{D}v(0)= 
	\bscp{\lam \hat{\delta}_1+(1-\lam)\hat{\delta}_2}{v_0},\quad\text{where $v_0:=\frac{v(\cdot)}{1+\abs{\cdot}^p}$,}
$$
whence $\lam \hat{\delta}_1+(1-\lam)\hat{\delta}_2\in H_x$ by Proposition~\ref{prop:pointconcOtoD}.
\end{proof}
\begin{prop}\label{prop:HxDenseInAx}
For every $x\in\partial\O$, $A_x$ is contained in the weak$^*$-closure of $H_x$.
\end{prop}
\begin{proof}%[Proof of Proposition~\ref{prop:HxDenseInAx}]
Let $v\in\ups$ and $a\in \RR$, define $v_0:=\frac{v(\cdot)}{1+\abs{\cdot}^p}\in\cR$, 
and suppose that $\scp{\mu}{v_0}=\scp{\mu}{v_0}_\infty\geq a$ for every $\mu\in H_x$.
By the Hahn-Banach theorem and the fact that $H_x$ is convex,
it suffices to show that in this case, we also have that $\scp{\pi}{v_0}=\scp{\pi}{v_0}_\infty\geq a$ for every $\pi\in A_x$.
We may assume w.l.o.g.~that $v(0)=0$ (replacing $v_0$ with $\tilde{v}_0(s):=v_0(s)-v_0(0)$ does not affect the assertion).
As before, we rely on Proposition~\ref{prop:pointconcOtoD} to work with sequences on $D$ instead of $\O$ in the definition of $H_x$.
For any bounded sequence $(u_n)\subset W_0^{1,p}(B_1;\RR^M)$ with support shrinking to the origin such that $\lim_{n\to \infty}\int_{D} w(\nabla u_{n}(y))\,dy$
exists for every $w\in\ups$ with $w(0)=0$, we have that
\begin{equation}\label{pHxAx-1}
	\alpha\leq \lim_{n\to \infty}\int_{D} v(\nabla u_{n}(y))\,dy. 
\end{equation}
If we fix one such sequence $(u_n)$, then for each $h>0$, the sequence $(u_{h,n})_n$, 
$$
	u_{h,n}(y):=\frac{1}{h}u_n(h y),
$$
is admissible, too, whence
\begin{equation}\label{pHxAx-2}
	\alpha\leq \lim_{n\to \infty}\int_{D} v(\nabla u_{h,n}(y))\,dy
	=\frac{1}{h^N} \lim_{n\to \infty}\int_{D} v(\nabla u_{n}(y))\,dy, 
\end{equation}
for every $h>0$. In the limit as $h\to\infty$, \eqref{pHxAx-2}
entails that $\alpha\leq 0$.

Next, we claim that $v$ is $p$-qscb at $x$. By Proposition~\ref{prop:lscconcnec}, it suffices to check that
$$
	0=\int_D v(0)\,dy\leq \lim_{n\to \infty}\int_{D} v(\nabla u_{n}(y))\,dy, 
$$
for every sequence $(u_n)\subset W_0^{1,p}(B_1;\RR^M)$ with support shrinking to the origin such that the limit above exists.
Suppose by contradiction that 
\begin{equation}\label{pHxAx-3}
	0> b:=\lim_{n\to \infty}\int_{D} v(\nabla u_{n}(y))\,dy, 
\end{equation}
for one such sequence $(u_n)$. Up to a subsequence,
(not relabeled), $(\nabla u_{n})_n$ generates a DiPerna-Majda measure, whence
$\lim_{n\to \infty}\int_{D} w(\nabla u_{n}(y))\,dy$
exists for every $w\in\ups$. Moreover, if we use this subsequence of $u_n$ to define $u_{m,n}$ as before, then
for every fixed $h>0$, the support is also shrinking to zero.
Hence, $u_n$ and $u_{h,n}$ are 
admissible in \eqref{pHxAx-1} and \eqref{pHxAx-2}, respectively, contradicting \eqref{pHxAx-3} if $h$ is sufficiently small.

In summary, we have shown that $a\leq 0$ and that $v$ is $p$-qscb, whence
$\scp{\pi}{v}\geq 0\geq a$ for every $\pi\in A_x$, by the definition of $A_x$.
\end{proof}
To complete the proof of Theorem~\ref{thm:suffboundary} in the present special case, 
we would have to show that $H_x$ is weak$^*$-closed. We skip this here as similar arguments are needed in the next step, anyway.
\subsection*{Step 2: General measures on the boundary}~\\
Ultimately, we reduce the general case to the first step by approximating a general measure with a finite sum of measures, 
each of which only charges one point on the boundary.
The construction of these is based on Lemma~\ref{lem:pqscbrotated}, which allows us to calculate a suitable average of a measure 
in a neighborhood of a point $x_0$ on the boundary while preserving \eqref{BGDM}:
\begin{prop}\label{prop:averagecm}
Let $1<p<\infty$, let $\O\subset\RR^N$ be open and bounded with boundary of class $C^1$, let $\nu(x)$ denote the outer normal to $\partial\O$ for $x\in\partial\O$, let $(\sigma,\hat{\nu})\in \DM$ and let $x_0\in\partial\O$. Moreover, let $\{R(x)\}_{x\in \partial\O}\subset SO(N)$ be a family of rotation matrices such that 
$x\mapsto R(x)$ is continuous and bounded on a set $\cU\subset \partial\O$ 
and for each $x\in \partial\O$, $\nu(x)=R(x) \nu(x_0)$.
Given a measurable set $E \subset \cU$ such that $\sigma(E)>0$, we define
$\hat{\eta}_{x_0}=\hat{\eta}_{x_0,E}\in \rca(\com)$ as the measure that satisfies
$$
	\int_\com v_0(s)\,\hat{\eta}_{x_0}(ds)
	%=\int_{\com\setminus \RR^{M\times N}} v_0(s)\hat{\eta}_{x_0}(ds)
	=\frac{1}{\sigma(E)}\int_{E} 
	\bscp{\hat{\nu}_x}{v_0(\,\cdot\, R_x^{-1})}\,\sigma(dx).
	%=\frac{1}{\sigma(E)}\int_{E} \int_{\com\setminus \RR^{M\times N}} v_0(s R_x^{-1})\hat{\nu}_x(ds)\,\sigma(dx).
$$
for every $v_0\in\cR$.
If $(\sigma,\hat{\nu})\in \BGDM$, then
\begin{equation} \label{pavcm-2}
	\begin{aligned}
	&0\leq \int_{\com\setminus \RR^{M\times N}} \frac{v(s)}{1+\abs{s}^p}\hat{\eta}_{x_0}(ds)\\
	&\text{for every $v\in\ups$ which is $p$-qscb at $x_0$.}
	\end{aligned}
\end{equation} 
\end{prop}
\begin{proof}
For every $v\in\ups$ which is $p$-qscb at $x_0$, $s\mapsto v(s R_x^{-1})$ is $p$-qscb at $x$ by Lemma~\ref{lem:pqscbrotated}. Hence, \eqref{BGDM}
implies \eqref{pavcm-2} by the definition of $\hat{\eta}_x$ with $v_0(s):=\frac{v(s)}{1+\abs{s}^p}$.
\end{proof}
Using this averaging procedure, we can weak$^*$-approximate general measures in $\BGDM$ by measures whose restriction to the boundary is supported on a finite number of points:
\begin{prop}\label{prop:weakstarapprox}
Let $(\sigma,\hat{\nu})\in \BGDM$.
Then for every $n\in\NN$, there exists a finite set $J(n)\subset \NN$ and $(\theta_n,\hat{\eta}_n)\in \BGDM$
such that $\theta_n|_\O=\sigma|_\O$, $\hat{\eta}_{n,x}=\hat{\nu}_x$ for $\sigma$-a.e.~$x\in\O$, 
$$
	\theta_n|_{\partial\O}=\sum_{j\in J(n)} a_{n,j}\delta_{x_{n,j}}|_{\partial\O},
$$ 
with suitably chosen points $x_{n,j}\in \partial\O$ and coefficients $a_{n,j}\geq 0$, $j\in J(n)$, 
where $\delta_{x_{n,j}}$ denotes the Dirac mass at $x_{n,j}$ in $\overline{\O}$,
and
$$
	\int_{\overline{\O}} \varphi(x) \bscp{\hat{\eta}_{n,x}}{v_0}\,\theta_n(dx)
	\underset{n\to\infty}{\To} 
	\int_{\overline{\O}} \varphi(x) \bscp{\hat{\nu}_{x}}{v_0}\,\sigma(dx)
$$
for every $\varphi\in C(\overline{\O})$ and every $v_0\in \cR$. 
\end{prop}
\begin{proof}
For each $n\in\NN$ cover $\RR^N$ with a family of pairwise disjoint cubes of side length $2^{-n}$, translates of $Q_{n,0}:=[0,2^{-n})^N$, and let 
$Q_{n,j}$, $j\in J(n)$, be the collection of those cubes $Q$ in the family that satisfy $\sigma(Q\cap \partial\O)>0$.
Moreover, for each $n$ and each $j\in J(n)$ let $E_{n,j}:=Q_{n,j}\cap \partial\O$, (arbitrarily) choose a point $x_{n,j}\in E_{n,j}$, 
and choose a family of rotations $(R_{n,j}(x))_{x\in E_{n,j}}\subset \RR^{N\times N}$ such that
$R_{n,j}(x_{n,j})=I$, $\nu(x)=R_{n,j}(x)\nu(x_{n,j})$ for every $x\in E_{n,j}$, where $\nu(x)$ denotes the outer normal at $x\in\partial\O$, 
$x\mapsto R_{n,j}(x)$ is continuous on $\overline{E}_{n,j}$ and 
\begin{equation}\label{pwsapprox-2}
	\sup_{j\in J(n)}\sup_{x\in E_{n,j}} \abs{R_{n,j}(x)^{-1}-I} \underset{n\to\infty}{\To} 0,
\end{equation}
which is possible since $\partial\O$ is of class $C^1$, at least if $n$ is large enough.
We define
$$
	\theta_n(dx):=\chi_{\partial\O}(x)\Big(\sum_{j\in J(n)} \sigma(E_{n,j})\delta_{x_{n,j}}(dx)\Big)+
	\chi_{\O}(x)\sigma(dx),
$$
and, for every $v_0\in\cR$,
$$
	\bscp{\hat{\eta}_{n,x}}{v_0}
	%=\int_{\com} v_0(s)\hat{\eta}_{x}(ds)
	%=\int_{\com\setminus \RR^{M\times N}} v_0(s)\hat{\eta}_{x_0}(ds)
	:=\left\{\begin{alignedat}[c]{2}
	&\frac{1}{\sigma(E_{n,j})}\int_{E_{n,j}} 
	\bscp{\hat{\nu}_y}{v_0(\,\cdot\, R_{n,j}(y)^{-1})}	\,\sigma(dy)
	&\quad&
	\text{if $x=x_{n,j}$},\\
	&\scp{\hat{\nu}_{x}}{v_0} &\quad&\text{elsewhere.}
	\end{alignedat}\right.
$$
Here, note that for $x\in \partial\O\setminus \{x_{n,j}\mid j\in J(n)\}$, the definition of 
$\hat{\eta}_{n,x}$ does not matter since
$\theta_n\big(\partial\O\setminus \{x_{n,j}\mid j\in J(n)\}\big)=0$.
Clearly, $(\theta_n,\hat{\eta}_n)\in \DM$, and $(\theta_n,\hat{\eta}_n)\in \BGDM$ by Proposition~\ref{prop:averagecm}.
Finally, observe that by \eqref{pwsapprox-2}, also using that
$\varphi$ is uniformly continuous on $\partial\O$ and that $v_0$ is uniformly continuous in the sense of \eqref{v0ltuc},
\begin{align*}
	&\int_{\partial\O} \varphi(x) \bscp{\hat{\eta}_{x}}{v_0}\,\theta_n(dx)\\
	&=\int_{\partial\O} \bigg(\sum_{j\in J(n)} \varphi(x_{n,j}) \chi_{E_{n,j}}(y)\bscp{\hat{\nu}_y}{v_0(\,\cdot\, R_{n,j,y}^{-1})}
	\bigg)\,\sigma(dy),\\
	&\to \int_{\partial\O}\varphi(y)\bscp{\hat{\nu}_y}{v_0}\,\sigma(dy)
\end{align*}
as $n\to\infty$.
\end{proof}
Our final ingredient is the following result of \cite{KaKru08a}, which states that subsets of $\DM$ defined 
by constraints on the generating sequences are always (sequentially) weak$^*$-closed (essentially because one can always choose an appropriate diagonal subsequence).
\begin{prop}[Lemma 3.3 in \cite{KaKru08a}]\label{prop:diag}
Let $S\subset L^p(\O;\RR^{M\times N})$ be an arbitrary bounded subset, and let $\DMS$ denote the subset of $\DM$ that consists of all DiPerna-Majda measures generated by a sequence $(U_n)\subset S$. If $(\sigma,\hat{\nu})\in \DM$ and 
$(\sigma_k,\hat{\nu}_k)$ is a sequence in $\DMS$ such that 
$(\sigma_k,\hat{\nu}_k)\rightharpoonup^* (\sigma,\hat{\nu})$, i.e.,
$$
	\int_{\overline{\O}} \varphi(x) \bscp{\hat{\nu}_k}{v_0}\,\sigma_k(dx)
	~\underset{k\to\infty}{\To}~
	\int_{\overline{\O}} \varphi(x) \bscp{\hat{\nu}}{v_0}\,\sigma(dx)
$$
for every $\varphi\in C(\overline{\O})$ and every $v_0\in\cR$, then $(\sigma,\hat{\nu})\in \DMS$.
%In particular, $\DMS$ is (sequentially) weak$^*$-closed in $\DM$.
\end{prop}
\begin{remark}\label{rem:diag}
Note that since both $C(\overline{\O})$ and $\cR$ are separable, the weak$^*$ topology is metrizable on bounded subsets of $\DM$, 
and weak$^*$-closed is equivalent to weak$^*$-sequentially closed.
Moreover, Proposition~\ref{prop:diag} also holds if $A$ is not bounded (e.g., $A:=\{\nabla u\mid u\in W^{1,p}(\O;\RR^M\}$):
If for each $k$, $(\sigma_k,\hat{\nu}_k)$ is generated by $(U_{k,n})_n\subset L^p(\Omega;A)$, then
$$
	\lim_{k\to\infty}\lim_{n\to\infty}\int_\O (1+\abs{U_{k,n}}^p)\,dx
	=\lim_{k\to\infty} \int_{\overline{O}} \bscp{\hat{\nu}_k}{1} \sigma_k(dx)
	=\int_{\overline{O}} \scp{\hat{\nu}}{1} \sigma(dx)<\infty.
$$	
Hence, passing to subsequences if necessary, we may assume that the generating sequences are equibounded,
and we can apply Proposition~\ref{prop:diag} with an appropriate bounded subset of $A$.
\end{remark}
We are now ready to prove the anticipated sufficient condition for gradient structure of the boundary part of a DiPerna-Majda measure, as defined in \eqref{DMsplit}:
\begin{proof}[Proof of Theorem~\ref{thm:suffboundary}]
Let $(\sigma,\hat{\nu})\in \BGDM$.
We have to show that $(\sigma_b,\hat{\nu}_b)\in \GDM$.
In view of Proposition~\ref{prop:GDMsplit}, we may assume w.l.o.g.~that $(\sigma,\hat{\nu})=(\sigma_b,\hat{\nu}_b)$, i.e., that 
$\sigma(dx)=dx$ in $\Omega$ and $\hat{\nu}_x(ds)=\delta_0(ds)$ for $x\in\Omega$. 
All the other DiPerna-Majda measure introduced below 
are understood to have this property as well, and for this reason, we will only define them on $\partial\Omega$.

By Proposition~\ref{prop:diag} and Remark~\ref{rem:diag}, it suffices to show that for each $n$, $(\theta_n,\hat{\eta}_n)\in \GDM$, where
$(\theta_n,\hat{\eta}_n)\in \BGDM$ is defined in Proposition~\ref{prop:weakstarapprox}. 
Recall that 
$$
	\begin{alignedat}{2}
	&\theta_{n}(dx)&&=\sum_{j\in J(n)} a_{n,j}\delta_{x_{n,j}}(dx)~~\text{on $\partial\Omega$},\\
	\end{alignedat}
$$
with a finite set $J(n)$, coefficients $a_{n,j}>0$ and points $x_{n,j}\in\partial\Omega$.
In particular, $\hat{\eta}_{n,x}$ is fully determined (i.e., for $\sigma_n$-a.e.~$x$) by $\hat{\eta}_{n,x_{n,j}}$, $j\in J(n)$.
Since $(\theta_n,\hat{\eta}_n)\in \BGDM$, 
for each $j\in J(n)$, we have that
$$
	\hat{\mu}_{n,j}:=a_{n,j}\hat{\nu}_{x_{n,j}}\in A_{x_{n,j}},
$$
by definition of $\BGDM$ and the set $A_x$ introduced in Step 1. 
By Proposition~\ref{prop:HxDenseInAx}, 
there exists a sequence $(\hat{\delta}_{n,j,k})_{k}\subset H_{x_{n,j}}$ 
which weak$^*$-converges to $\hat{\delta}_{n,j}$ in $\rca(\com)$ as $k\to\infty$. 
Accordingly, the corresponding sequence of DiPerna-Majda measures $(\theta_{n,k},\hat{\eta}_{n,k})_k \subset \DM$, defined by
$$
	\begin{alignedat}{2}
	&\theta_{n,k}(dx)&&:=\sum_{j\in J(n)}\hat{\delta}_{n,j,k}(\com)\delta_{x_{n,j}}(dx)~~\text{on $\partial\Omega$},\\
	&\hat{\eta}_{n,k,x}(ds)&&:=\frac{1}{\hat{\delta}_{n,j,k}(\com)}\hat{\delta}_{n,j,k}(ds)~~\text{if $x=x_{n,j}$ for some $j\in J(n)$,}
	\end{alignedat}
$$
weak$^*$-converges to $(\theta_n,\hat{\eta}_n)$ in $\DM$. Hence, by Propositon~\ref{prop:diag} and Remark~\ref{rem:diag}, it suffices to show that $(\theta_{n,k},\hat{\eta}_{n,k})\in \GDM$. 

By definition of $H_{x_{n,j}}$, for each $j\in J(n)$ there exists a bounded 
sequence $(u_{j,m})_m\subset W^{1,p}(\O;\RR^M)$ (also depending on $n$ and $k$) with support shrinking to $x_{n,j}$
such that $(\nabla u_{j,m})_m$ generates $\hat{\delta}_{n,j,k}$, which implies that
\begin{equation}\label{psuff-1}
	\begin{aligned}
	\lim_{m\to\infty}\int_\O \varphi(x) \big[v(\nabla u_{j,m}(x))-v(0)\big]\,dx
	&=\varphi(x_{n_j})\bscp{\hat{\delta}_{n,j,k}}{v_0}\\
	&=\int_{\partial\O} \varphi(x) \bscp{\hat{\eta}_{n,k,x}}{v_0} \,\delta_{x_{n,j}}(dx),
	\end{aligned}
\end{equation}
for every $\varphi\in C(\overline{\O})$ and every $v\in \ups$, where $v_0:=\frac{v(\cdot)}{1+\abs{\cdot}^p}$. 
We define
$$
	u_{m}:=\sum_{j\in J(n)} u_{j,m}\in W^{1,p}(\O;\RR^M),
$$
which is a sum of functions with pairwise disjoint support (at least for large $m$).
Summing over $j$ in \eqref{psuff-1} yields that
\begin{equation}\label{psuff-2}
	\begin{aligned}
	&\lim_{m\to\infty}\int_\O \varphi(x) \big[v(\nabla u_{m}(x))-v(0)\big]\,dx\\
	&=\lim_{m\to\infty}\sum_{j\in J(n)}\int_\O \varphi(x) \big[v(\nabla u_{j,m}(x))-v(0)\big]\,dx\\
	&=\int_{\partial{\O}} \varphi(x) \bscp{\hat{\eta}_{n,k,x}}{v_0} \,\theta_{n,k}(dx).
	\end{aligned}
\end{equation}
Since $\theta_{n,k}(dx)=dx$ in $\O$, $\hat{\eta}_{n,k,x}(ds)=\delta_0(ds)$ for $x\in \O$ and $v_0(0)=v(0)$, 
\begin{equation}\label{psuff-3}
	\int_\O \varphi(x)v(0)\,dx=\int_{\O} \varphi(x) \scp{\hat{\eta}_{n,k,x}}{v_0} \,\theta_{n,k}(dx).
\end{equation}
Plugging \eqref{psuff-3} into \eqref{psuff-2}, we obtain that
\begin{equation*}
	\begin{aligned}
	\lim_{m\to\infty}\int_\O \varphi(x) v(\nabla u_{m}(x))\,dx
	&=\int_{\overline\O} \varphi(x) \scp{\hat{\eta}_{n,k,x}}{v_0} \,\theta_{n,k}(dx).
	\end{aligned}
\end{equation*}
This means that $(\nabla u_{m})_m$ generates $(\theta_{n,k},\hat{\eta}_{n,k})$, and consequently, 
$(\theta_{n,k},\hat{\eta}_{n,k})\in \GDM$.
\end{proof}

\section{Proof of the relaxation result}\label{sec:relax}

\begin{proof}[Proof of Theorem~\ref{relaxation}]~\\
{\bf (i):}
Suppose that $\{u_n\}$ already realizes the $\liminf$. Fix $x\in\bar\O$, apply Theorem~\ref{thm:GDMchar} to $v:=h(x,\cdot)$ and integrate over $\bar\O$. In view of  (ii), (iii), and (iv) in  Theorem~\ref{thm:GDMchar} we get 
\begin{equation*}
\begin{aligned}
\int_\O Qh(x,\nabla u(x))\,dx  
~\le~ & \int_\O\int_{\com}\frac{h(x,s)}{1+|s|^p}\nu_x(ds)d_\sigma(x)\,dx\nonumber\\
&+ \int_{\bar\O}\int_{\com\setminus \RR^{M\times N}} \frac{h(x,s)}{1+|s|^p}\nu_x(ds)\sigma_s(dx)\\
~=~&\int_{\bar\O}\int_\com \frac{h(x,s)}{1+|s|^p}\nu_x(ds)\sigma(dx)\nonumber\\
~=~&\lim_{n\to\infty}\int_\O h(x,\nabla u_n(x))\,dx\nonumber\ . 
\end{aligned}
\end{equation*}
Here, note that for each $x\in \overline{\O}$, either $Qh(x,\cdot)\equiv-\infty$ or $Qh(x,\cdot)>-\infty$. In the former case, we cannot use Theorem~\ref{thm:GDMchar} (ii), but the corresponding estimate above then becomes trivial. 

%{\bf (ii):}
%The existence of $\{\tilde u_n\}$ with the  approximation property  follows from \cite[Th.~9.8]{Da89B}.
%HOW TO GET THE GROWTH CONDITION FOR $Qh$, IN PRESENCE OF THE DEPENDENCE ON $x$? 

\noindent{\bf (ii):} %We only provide a proof for the case $QH(\tilde{u})>-\infty$; the argument for $QH(\tilde{u})=-\infty$ is similar. NO, NOT REALLY!!!!!!
Below, we use the shorthand $f\wedge g:=\max\{f,g\}$, pointwise for real-valued functions. 
Let $\eps>0$. For $m\in \NN$ set
$$
	V_m(x,s):=\frac{1}{m}\abs{s}^p-m\quad\text{for $x\in\Omega$ and $s\in\RR^{M\times N}$. }
%	F:=\mysetr{x\in\O}{Qh(x,\cdot)>\infty}\quad\text{and}\quad
%	V_m(x,s):=\left\{\begin{alignedat}[c]{2}
%	&\frac{1}{m}\abs{s}^p-m &\quad&\text{if $Q h(x,\cdot)>-\infty$,}\\
%	&-m &\quad&\text{if $Q h(x,\cdot)=-\infty$}
%	\end{alignedat}\right.
$$
Since $h\wedge V_m$ is $p$-coercive for each $m$, there exists a sequence $(\tilde{u}_{m,n})_n\subset \tilde{u}+W_0^{1,p}(\O;\RR^M)$ such that
$\tilde{u}_{m,n}\rightharpoonup \tilde{u}$ in $W^{1,p}$ as $n\to \infty$ and
$$
	\int_\O (V_m\wedge h)(x,\nabla \tilde{u}_{m,n})\,dx\underset{n\to\infty}{\To} \int_\O Q(V_m\wedge h)(x,\nabla\tilde{u})\,dx,
$$
by the standard relaxation result, see for instance \cite[Th.~9.8]{Da89B}. 
By the trivial estimate $h\leq V_m\wedge h$, this implies that
\begin{equation}\label{trelaxlimsup0}
	\limsup_{n\to\infty} \int_\O h(x,\nabla \tilde{u}_{m,n})\,dx
	~\leq ~ \int_\O Q(V_m\wedge h)(x,\nabla \tilde{u})\,dx
\end{equation}
{\bf \underline{Case 1}: $\mbf{\abs{E}=0}$, i.e., $\mbf{Qh(x,\cdot)>-\infty}$ for a.e.~$\mbf{x\in\O}$.}\\ 
We split the integral on the right hand side of \eqref{trelaxlimsup0} into integrals over the set
$$
	G(m,\eps):=\mysetr{x\in\O}{Q(V_m\wedge h)(x,\nabla\tilde{u}(x))\leq Q h(x,\nabla\tilde{u}(x))+\frac{\eps}{2\abs{\O}}},
$$
and its complement. Thus, we get that
\begin{equation}\label{trelaxlimsup1}
	\begin{aligned}
	&\limsup_{n\to\infty} \int_\O h(x,\nabla \tilde{u}_{m,n})\,dx\\
	&\quad\leq ~ \int_{G(m,\eps)} Q h(x,\nabla\tilde{u})\,dx+\frac{\eps}{2}+ \int_{\O\setminus G(m,\eps)} C(\nabs{\nabla \tilde{u}}^p+1)\,dx,
	\end{aligned}
\end{equation}
where we also used that $h(x,s)\leq C(\nabs{s}^p+1)$ for some constant $C>0$, and $Q(V_m\wedge h)$ inherits this upper bound, at least for large $m$.
Observe that $G(m,\eps)\subset G(m+1,\eps)$. Moreover, by the representation formula for quasiconvex envelopes \eqref{Qfrep}, 
there exists a function $\varphi\in W_0^{1,\infty}(\O;\RR^M)$ such that 
$$
	\frac{1}{\abs{\O}}\int_\O h(x,\nabla \tilde{u}(x)+\nabla \varphi(y))\,dy\leq Qh(x,\nabla \tilde{u}(x))+\eps.
$$
In particular, $x\in G(m,\eps)$ if $m$ is large enough so that $(V_m\wedge h)(x,\nabla \tilde{u}(x)+s)= h(x,\nabla \tilde{u}(x)+s)$ for every 
$\abs{s}\leq \norm{\nabla \varphi}_\infty$.
Hence, we have that $\bigcup_{m\in\NN}G(m,\eps)=\Omega$,
and
\begin{equation}\label{trelaxlimsup2}
	\int_{G(m,\eps)} Q h(x,\nabla\tilde{u})\,dx\underset{m\to\infty}{\To} \int_{\O} Q h(x,\nabla\tilde{u})\,dx
\end{equation} 
by monotone/dominated convergence (for the negative and the positive part of the integrand, respectively). For the same reason, 
\begin{equation}\label{trelaxlimsup3}
	\int_{\O\setminus G(m,\eps)} C(\nabs{\nabla \tilde{u}}^p+1)\,dx\underset{m\to\infty}{\To} 0.
\end{equation} 
Combined,\eqref{trelaxlimsup1}--\eqref{trelaxlimsup3} imply that there exists $M(\eps)\in\NN$ such that
$$
	\limsup_{n\to\infty} \int_\O h(x,\nabla \tilde{u}_{m,n})\,dx\leq \int_{\O} Q h(x,\tilde{u})\,dx+\eps,
$$
for every $m\geq M(\eps)$.\\
{\bf \underline{Case 2}: $\mbf{\abs{E}>0}$ for $\mbf{E=\{x\in\O\mid \mbf{Qh(x,\cdot)=-\infty}\}}$.}\\
We now define
$$
	G(m,\eps):=\mysetr{x\in E}{Q(V_m\wedge h)(x,\nabla\tilde{u}(x))\leq -\frac{4}{\eps\abs{E}}}
$$
Once more splitting the integral on the right hand side of \eqref{trelaxlimsup0}, we now get that
\begin{equation}\label{trelaxlimsup5}
	\begin{aligned}
	&\limsup_{n\to\infty} \int_\O h(x,\nabla \tilde{u}_{m,n})\,dx\\
	&\quad\leq ~ \int_{G(m,\eps)} -\frac{4}{\eps\abs{E}} \,dx+ \int_{\O\setminus G(m,\eps)} C(\nabs{\nabla \tilde{u}}^p+1)\,dx.
	\end{aligned}
\end{equation}
As before, $G(m,\eps)\subset G(m+1,\eps)$, and due to \eqref{Qfrep}, $\bigcup_{m\in\NN} G(m,\eps)=E$. 
For $m$ large enough, we thus have that $\abs{G(m,\eps)}\geq \frac{1}{2} \abs{E}$, and consequently,
\begin{equation*}%\label{trelaxlimsup6}
	\begin{aligned}
	\limsup_{n\to\infty} \int_\O h(x,\nabla \tilde{u}_{m,n})\,dx
	&\leq ~ -\frac{2}{\eps}+ \int_{\O} C(\nabs{\nabla \tilde{u}}^p+1)\,dx.
	\end{aligned}
\end{equation*}
This implies the assertion if $\eps$ is sufficiently small.
\end{proof}

\bigskip

{\bf Acknowledgment:} This work  was initiated during the stay of MK at the University of Cologne. Its support and hospitality is gratefully acknowledged.

\bibliographystyle{plain}
\bibliography{DMbd}

\end{sloppypar}

\end{document}